\documentclass[12pt, eqno]{article}

\usepackage{amsmath,amsthm,amssymb,enumerate}
\usepackage{epsfig}
\usepackage{amssymb}
\usepackage{amsmath}
\usepackage{amssymb}
\usepackage{amsmath,amsthm}
\usepackage[latin1]{inputenc}
\usepackage[T1]{fontenc}
\usepackage{path}
\usepackage{ae,aecompl}
\usepackage{amsfonts}
\usepackage{amsxtra}
\usepackage{bbm,euscript,mathrsfs}
\usepackage{color}
\usepackage{lipsum}

\newcommand\blfootnote[1]{%
  \begingroup
  \renewcommand\thefootnote{}\footnote{#1}%
  \addtocounter{footnote}{-1}%
  \endgroup
}

\allowdisplaybreaks

\usepackage{amsfonts}
\usepackage{amsxtra}

\usepackage[frak,hyperref,theorem]{paper_diening}

\newcommand{\Div}{\divergence}
\newcommand{\ep}{\varepsilon}
\newcommand{\R}{\mathbb R}
\newcommand{\N}{\mathbb N}

\newcommand{\E}{\mathbb E}
\newcommand{\p}{\mathbb P}

\newcommand{\F}{\mathfrak F}

\newcommand{\dd}{\mathrm{d}}
\newcommand{\dx}{\,\mathrm{d}x}
\newcommand{\dt}{\,\mathrm{d}t}
\newcommand{\dxt}{\,\mathrm{d}x\,\mathrm{d}t}
\newcommand{\ds}{\,\mathrm{d}\sigma}
\newcommand{\dxs}{\,\mathrm{d}x\,\mathrm{d}\sigma}

\newcommand{\UU}{\mathfrak{U}}
\newcommand{\V}{\mathcal{V}}
\newcommand{\eps}{\varepsilon}
\newcommand{\dif}{\mathrm{d}}

\newcommand{\mf}{\mathfrak{F}}

\newcommand{\prst}{\mathbb{P}}

\newcommand{\mt}{\mathbb T^3}

\newcommand{\bxi}{\bfxi}
\newcommand{\T}{\mathbb T^3}

\newcommand{\StoB}{\left(\Omega, \mathfrak{F},\left(\mathfrak{F}_t \right)_{t \geq 0},  \mathbb{P}\right)}

\newcommand{\db}[1]{\textcolor[rgb]{0.00,0.00,0.00}{  #1}}

\numberwithin{equation}{section}
\numberwithin{theorem}{section}

\begin{document}
\begin{center}
		\Large Dissipative solutions to the stochastic Euler equations\\[4ex]
		\normalsize D. Breit and T. C. Moyo
\end{center}
  \begin{abstract}
We study the three-dimensional incompressible Euler equations subject to stochastic forcing.
We develop a concept of dissipative martingale solutions, where the nonlinear terms are described by generalised Young measures.
We construct these solutions as the vanishing viscosity limit
of solutions to the corresponding stochastic Navier--Stokes equations. This requires a refined stochastic compactness method incorporating the generalised Young measures.\\
 Our solutions satisfy a form of the energy inequality which gives rise to a weak-strong uniqueness result (pathwise and in law). A dissipative martingale solution coincides (pathwise or in law) with the strong solution as soon as the latter exists.
  \end{abstract}

\blfootnote{  \vspace{-3ex}
\par\noindent {\it Adress:
Heriot-Watt University, Department of Mathematics,
EH14 4AS Riccarton, Edinburgh, UK}
}

\section{Introduction}
We are interested in the stochastic Euler equations describing the motion
of an incompressible inviscid fluid in the three-dimensional torus $\mt$. The flow is described by the velocity field $\bfu:Q\rightarrow\R^3$, $Q=(0,T)\times\mt$, and the pressure $\pi:Q\rightarrow \R$ and the equations in question read as
\begin{align}\label{eq:Eulera}
\left\{\begin{array}{rc}
\dd\bfu=-(\nabla\bfu)\bfu\dt-\nabla\pi\dt+\Phi\dd W
& \mbox{in $Q$,}\\
\Div \bfu=0\qquad\qquad\qquad& \mbox{in $Q$,}
%\bfu(0)=\bfu_0\qquad\qquad\qquad& \mbox{in $\mt$,}
\end{array}\right.
\end{align}
subject to periodic boundary conditions for $\bfu$.
The first equation in \eqref{eq:Eulera} is forced by a cylindrical Wiener process $W$ and $\Phi$ is  a Hilbert--Schmidt operator, see Section \ref{subsec:prob} for details. Stochastic forces in the equations of motion are frequently used to model phenomena in turbulent flows at high Reynolds number, see e.g. \cite{Ei,No,Vi}.\\
As in the deterministic case smooth solutions to \eqref{eq:Eulera} are only known to exist locally in time, see \cite{GHVic,Kim,MiVa}. The life space of these solutions is an a.s. positive stopping time.
While better results are known in the two-dimensional situation, cf. \cite{BeFl,BrPe,CaCu,Kim2}, the existence and uniqueness of global strong solutions is a major open problem. In the deterministic case a series of counter examples concerning uniqueness of solutions to the Euler equations have been accomplished recently. These solutions are called wild solutions and are constructed by the method of convex integration pioneered by the work of De Lellis and Sz\'ekelyhidi \cite{DelSze2,DelSze3}.
As shown in \cite{BFHconvex} stochastic forces do not seem to change the situation.\\
In view of these examples one may expect that singularities occur in the long-run and that solutions are not unique. A natural approach to deal with such situations is the concept of measure-valued solutions as introduced by Di Perna and Majda \cite{DiMa} (see also \cite{Di}). These solutions are constructed by compactness methods and the nonlinearities are described by
generalised Young measures. A generalised Young measure is a triplet $\mathcal V=(\nu_{t,x},\nu_{t,x}^\infty,\lambda)$ consisting of the oscillation measure $\nu_{t,x}$ (a parametrised probability measure), the concentration measure $\lambda$ (a non-negative Radon measure) and
the concentration angle $\nu_{t,x}^\infty$ (a parametrised probability measure on the unit sphere). The convective term can be written as the space-time distribution
$$\Div\big\langle\nu_{t,x},\bfxi\otimes\bfxi\big\rangle\dxt+\Div\big\langle\nu_{t,x}^\infty,\bfxi\otimes\bfxi\big\rangle\,\dif \lambda.$$
This is the only available framework which allows us to obtain (for any given initial datum) the long-time existence of solutions, which comply with basic physical principles such as the dissipation of energy (the existence of weak solutions for any initial datum, which violate the energy inequality, has been shown in \cite{W}). The energy inequality implies a weak-strong uniqueness principle for measure-valued solutions as shown in \cite{BrDeSz}: A measure-valued solution coincides with the strong solution as soon as the strong solution exists.\\
\db{While all these results concern the deterministic case, there is strong interests to study measure-valued solutions to the three-dimensional stochastic Euler equations \eqref{eq:Eulera} in order to grasp its long-term dynamics. The first result is the existence of martingale solutions in
\cite{Kim3}, where the equations of motion are understood in the measure-valued sense. These solutions are weak in the probabilistic sense, that is the underlying probability space as well as the driving Wiener process are not a priori given but become an integral part of the solution. Such a concept is common for stochastic evolutionary problems when uniqueness is not available. It is classical for finite dimensional problems and has also been applied to various stochastic partial differential equations, in particular in fluid mechanics (see, for instance, \cite{BrHo,BrM,CaCu1,debussche1,FlGa}).
Unfortunately, the solutions constructed in \cite{Kim3} do only satisfy a form of energy estimate in expectation with an unspecified constant $C$ on the right-hand side, rather than an energy inequality as in the deterministic case. This is not enough
to conclude with a weak-strong uniqueness principle which one should require for any reasonable notion of generalized solution, cf. \cite{Li1}.\\
%We approximate \eqref{eq:Euler} by a sequence of Navier--Stokes equations with vanishing viscosity and use a refined stochastic compactness method. It is based on Jakubowski's extension
%of Skorokhod's representation theorem \cite{jakubow} in order to incorporate the generalised Young measures.\\
The aim of this paper is to close this gap and to develop a concept of measure-valued martingale solutions to \eqref{eq:Eulera} which satisfy a suitable energy inequality.
 The solution are called \emph{dissipative} and our energy inequality can be described as follows:
 If $\mathcal V=(\nu_{t,x},\nu_{t,x}^\infty,\lambda)$ is the generalised Young measure associated to the solution, then the kinetic energy
$$E_t=\frac{1}{2}\int_ {\mt}\big\langle\nu_{t,x},|\bfxi|^2\big\rangle\dx+\frac{1}{2}\lambda_t(\mt),\quad \lambda=\lambda_t\otimes \mathcal L^1,$$
satisfies
\begin{align*} 
 E_{t^+}
& \leq E_{s^-}+\frac{1}{2} \int_s^t \|\Phi\|_{L_2}^2   \dd\tau
+  \int_s^t\int_{\mt} \bfu\cdot\Phi\dx \, {\rm d}W,\quad E_{0^-}=\frac{1}{2}\int_{\mt}|\bfu(0)|^2\dx,
\end{align*}
$\mathbb P$-a.s. for any $0\leq s<t,$
see Definition \ref{def:soleuler} for the precise formulation.}
%This is a key ingredient for further applications.
In the deterministic case the energy is non-increasing and non-negative such that the left- and right-sided limits $E_{t^-}$ and $E_{t^+}$ exist for any $t$. In the stochastic case one has instead that the difference between the energy
and a continuous function is increasing and that both are pathwise bounded such that the same conclusion holds, see also Remark \eqref{remark:energy}. Nevertheless, some care is required to implement this idea within the stochastic compactness method, see Section 
\ref{subsec:new}.
With the energy inequality just described at hand we are able to analyse the weak-strong uniqueness property of \eqref{eq:Eulera}. In a pathwise
approach we prove that a dissipative martingale solution agrees with the strong solution if both exist on the same probability space. This is reminiscent of the deterministic analysis in \cite{BrDeSz}. For this it is crucial that the energy inequality discussed above holds for any time $t$ in order to work with stopping times. A more realistic assumption is that the probability spaces, on which both solutions exit, are distinct.
In this situation we prove that the probability laws of the weak and the strong solution coincide. This is based on the classical Yamada-Watanabe argument, where a product probability space is constructed. Thereby, the weak-strong uniqueness in law can be reduced to the pathwise weak-strong uniqueness already obtained. We face several difficulties due to the fact that \eqref{eq:Eulera} is infinite-dimensional and, in particular, due to the non-separability of the space of generalised Young measures. \\
The paper is organised as follows. In Section \ref{sec:prelim} we present some preliminary material. In particular, we introduce the set-up for generalised Young measure, present the concept of random distributions from \cite{BFHbook} (in order to define progressive measurability for stochastic processes which are only equivalence classes in time) and prove an infinite dimensional It\^{o}-formula which is appropriate for our purposes. Finally, we collect some known material on the stochastic Navier--Stokes equations. The latter will be needed to approximate the stochastic Euler equations. In Section \ref{sec:DMS} we introduce the concept of dissipative martingale solutions
and prove their existence. \db{As in \cite{Kim3} we approximate \eqref{eq:Eulera} by a sequence of Navier--Stokes equations with vanishing viscosity and use a refined stochastic compactness method (based on Jakubowski's extension of Skorokhod's representation theorem \cite{jakubow})}. Section \ref{sec:weakstrong} is dedicated to weak-strong uniqueness.

\section{Mathematical framework}
\label{sec:prelim}
In this section we present various preliminaries on generalised Young measures, random variables and stochastic integration. Moreover, we collect some known material on the stochastic Navier--Stokes equations.
\subsection{Generalised Young measures}
\label{sec:GYM}
We denote by $\mathscr M$ the set of Radon measures, by $\mathscr M^+$ the set of non-negative Radon measures and by
$\mathscr P$ the set of probability measures. In our application there will be usually defined on a parabolic cylinder $Q_T=(0,T)\times\mt$. We will only use the integrability index $p=2$. Also, without further mentioning it, we will exclusively deal with generalised Young measures generated by sequences of functions with values in $\R^3$. A generalised Young measure is defined as follows.
\begin{definition}
A quantity $\mathcal V=(\nu_{t,x},\nu_{t,x}^\infty,\lambda)$ is called generalised Young measure provided
\begin{enumerate}[(a)]
\item $\nu_{t,x}\in L^\infty_{w^*}(Q_T;\mathscr P(\R^3))$ is a parametrised probability measure on $\R^3$;
\item $\lambda\in \mathscr M^+(\overline Q_T)$ is a non-negative Radon measure;
\item $\nu_{t,x}^\infty\in L^\infty_{w^*}(Q_T,\lambda;\mathscr P(\mathbb S^2))$ is a parametrised probability measure on $\mathbb S^2$;
\item We have $\int_{Q_T}\langle\nu_{t,x},|\bfxi|^2\rangle\dxt<\infty$.
\end{enumerate}
We denote the space of all generalised Young measure by $Y_2(Q_T)$.
\end{definition}
In particular, any Radon measure $\mu\in\mathscr M(Q_T)$ can be represented by a generalized Young measure by setting $\mathcal V=\big(\delta_{\mu^a(t,x)},\frac{\dd \mu^s}{\dd |\mu^s|},|\mu^s|\big),$
where $\mu=\mu^a\,\dd \mathscr L^n+\mu^s$ is the Radon-Nikod\'ym decomposition of $\mu$.
We consider now all Carath\'eodory functions $f:Q_T\times\R^3\rightarrow \R$ such that the recession function
\begin{align*}
f^\infty(t,x,\bfxi):=\lim_{s\rightarrow\infty}\frac{f(t,x,s\bfxi)}{s^2}
\end{align*}
is well-defined and continuous on $\overline Q_T\times \mathbb S^2$ (which implies that $f$ grows at most quadratically in $\bfxi$). We denote by $\mathcal G_2(Q_T)$ the space of all such functions. We say a sequence $\{\mathcal V^n\}=\{(\nu_{t,x}^n,\nu_{t,x}^{\infty,n},\lambda^n)\}$ converges weakly* in $Y_2(Q_T)$ to some  $\mathcal V=(\nu_{t,x},\nu_{t,x}^{\infty},\lambda)\in Y_2(Q_T)$ provided
\begin{align*}
&\langle\nu_{t,x}^n,f(\bfxi)\rangle\dxt+\langle\nu_{t,x}^{\infty,n},f^\infty(\bfxi)\rangle\,\dd\lambda^n\\
&\quad\rightharpoonup^\ast \langle\nu_{t,x},f(\bfxi)\rangle\dxt+\langle\nu_{t,x}^{\infty},f^\infty(\bfxi)\rangle\,\dd\lambda\quad\text{in}\quad \mathscr M(Q_T)
\end{align*}
for all $f\in \mathcal G_2(Q_T)$, that is
\begin{align*}
\int_{Q_T}&\varphi\langle\nu_{t,x}^n,f(\bfxi)\rangle\dxt+\int_{Q_T}\varphi\langle\nu_{t,x}^{\infty,n},f^\infty(\bfxi)\rangle\,\dd\lambda^n\\
&\rightarrow \int_{Q_T}\varphi\langle\nu_{t,x},f(\bfxi)\rangle\dxt+\int_{Q_T}\varphi\langle\nu_{t,x}^{\infty},f^\infty(\bfxi)\rangle\,\dd\lambda
\end{align*}
for all $\varphi\in C(\overline Q_T)$. Here $\bfxi\in\R^3$ denotes the corresponding dummy-variable. 
 The space $\mathcal G_2(Q_T)$ is a separable Banach space together with the norm
\begin{align*}
\|f\|_{\mathcal G_2(Q_T)}:=\sup_{(t,x)\in Q_T,\,\bfxi\in B_1(0)}(1-|\bfxi|)\Big|f\Big(t,x,\frac{\bfxi}{1-|\bfxi|}\Big)\Big|
\end{align*} 
  and $Y_2(Q_T,\R^n)$ is a subspace of its dual. Consequently, $Y_2(Q_T,\R^n)$ together with the weak* convergence introduced above is a quasi-Polish space.\\
A topological space $(X,\tau)$ is called quasi-Polish space if there is a countable family 
\begin{align}\label{eq:quasi}
\big\{f_n:X\rightarrow[-1,1];\,n\in\N\big\}
\end{align}
of continuous functions that separates points. In particular, separable Banach spaces endowed with the weak topology and dual spaces of separable Banach spaces are quasi-Polish spaces. Since we are interested in the long-time behaviour we also define
\begin{align*}
Y_2^{\mathrm loc}(Q_\infty)=\big\{\mathcal V:\,\mathcal V\in Y_2(Q_T)\,\forall T>0\big\}.
\end{align*}
Since the topology on $Y_2^{\mathrm loc}(Q_\infty)$ is generated by the topologies on $Y_2(Q_T)$ in the sense that
\begin{align*}
\mathcal{V}^n\rightharpoonup^\ast\mathcal V\quad\text{in}\quad Y_2^{\mathrm loc}(Q_\infty)\quad\Leftrightarrow\quad \mathcal{V}^n\rightharpoonup^\ast\mathcal V\quad\text{in}\quad Y_2(Q_T) \quad\forall T>0,
\end{align*}
it is clear that $Y_2^{\mathrm loc}(Q_\infty)$ is a quasi-Polish space as well.
\\
%Let us conclude this subsection with the compactness criterion for generalised Young measures (see \cite{DiMa} and \cite{AB}, see also \cite[Cor. 2]{KrRi} for the corresponding $L^1$-version).
%\begin{proposition}\label{prop:GYM}
%Let $\{\mathcal V^n\}=\{(\nu_{t,x}^n,\nu_{t,x}^{\infty,n},\lambda^n)\}$ be a sequence of generalised Young measures such that
%\begin{enumerate}[(a)]
%\item $\int_{Q_T}\langle\nu_{t,x}^n,|\bfxi|^2\rangle\dxt$ stay uniformly bounded;
%\item $\lambda^n(\overline Q_T)$ stays uniformly bounded.
%\end{enumerate}
%Then $\{\mathcal V^n\}$ is sequentially relatively compact in $Y_2(Q_T)$, that is there is a subsequence (not-relabelled) such that $\mathcal V^n\rightharpoonup^* \mathcal V$ in $Y_2(Q_T)$
%for some $\mathcal V\in Y_2(Q_T)$.
%\end{proposition}
%In particular, bounded sequence in $L^2(Q_T)$ are relatively compact in $Y_2(Q_T)$, where 
We can embed
$L^2(Q_T)$ into $Y_2(Q_T)$ via the inclusion
\begin{align}\label{eq:L2Y}
L^2(Q)\ni u\mapsto (\delta_{u(t,x)},0,0)\in Y_2(Q_T).
\end{align}
\db{By the Alaoglu-Bourbaki theorem, for any $L>0$ there is a compact subset $\mathcal K_L$ of $\mathcal G_2(Q_T)^*$ such that
\begin{align*}
\{(\delta_{u(t,x)},0,0)\in Y_2(Q_T):\|u\|_{L^2(Q_T)}\leq L\}\subset \mathcal K_L.
\end{align*}
Since $Y_2(Q_T)$ is weak* closed in $\mathcal G_2(Q_T)^*$ we conclude that $\mathcal K_L\cap Y_2(Q_T)$ compact, where clearly
\begin{align}\label{eq:L2Y'}
\{(\delta_{u(t,x)},0,0)\in Y_2(Q_T):\|u\|_{L^2(Q_T)}\leq L\}\subset \mathcal K_L\cap Y_2(Q_T).
\end{align}}
It is also useful to identify a generalised Young measure with a space-time distribution: For $\mathcal V=(\nu_{t,x},\nu_{t,x}^{\infty},\lambda)\in Y_2(Q_T)$ we define
\begin{align}\label{eq:2402b}
\begin{aligned}
C_c^\infty(Q_T\times\R^3)^2\ni (\psi,\varphi)&\mapsto\int_{Q_T}\int_{\R^3}\psi(t,x,\bfxi)\,\dd\nu_{t,x}(\bfxi)\dxt\\&+\int_{Q_T}\int_{\R^3}\varphi(t,x,\xi)\,\dd\nu_{t,x}^{\infty}(\bfxi) \,\dd\lambda(t,x).
\end{aligned}
\end{align}
As we will study probability laws on $Y_2(Q_T)$ we need a $\sigma$-field. A suitable candidate is the $\sigma$-algebra generated by the functions $\{f_n\}$ from \eqref{eq:quasi}, that is we set
\begin{align}\label{eq:sigmaY}
\mathscr B_{Y}:=\sigma\bigg(\bigcup_{n=1}^\infty\sigma(f_1)\bigg).
\end{align}
%Note that, using the identification from \eqref{eq:L2Y}, we have
%\begin{align}\label{eq:sigmaL2Y}
%L^2(Q_T)\in \mathscr B_{Y}.
%\end{align}

\subsection{Random distributions}
\label{sec:prelimsstoch}
Let $Q_T = (0,T) \times \mt$. Let $(\Omega,\mf,(\mf_t)_{t\geq0},\prst)$ be a complete stochastic basis with a Borel probability measure $\prst$ and a right-continuous filtration
$(\mf_t)$. For a measurable space
$(X,\mathcal{A})$ an $X$-valued random variable is a measurable mapping
$
\bfU : (\Omega,\mathfrak{F}) \to (X,\mathcal{A}).
$
We denote by $\sigma(\bfU)$ the smallest $\sigma$-field with respect to which $\bfU$ is measurable, that is
$$
\sigma(\bfU):= \big\{\{\omega\in\Omega;\,\bfU(\omega)\in A\};\, A\in\mathcal{A}\big\}.
$$
In order to deal with oscillations and concentrations in the convective term of approximate solutions to the stochastic Euler equations we have to deal with generalised Young measures (as introduced in the previous subsection) and hence we need to study
mappings $\bfU:\Omega\rightarrow Y_2(Q_T)$.  
Such an object is not a stochastic process in the classical sense as it is only defined a.e. in time.
Consequently, it becomes ambiguous to speak about progressive measurability. To overcome such problems the concept of random distributions has been introduced in \cite{BFHbook}[Chap. 2.2] to which we refer to for more details.
\begin{definition} \label{RDD1}
Let $( \Omega, \mathfrak{F}, \prst )$ be a complete probability space. A mapping
\[
\bfU: \Omega \to \big(C^\infty_c(Q_T)\big)'
\]
is called \emph{random distribution} if $\langle\bfU, \bfphi\rangle:\Omega\to\R$ is a measurable function for any $\bfphi \in C^\infty_c(Q_T)$.
\end{definition}
In order to introduce a concept of progressive measureability we consider  the $\sigma$-field  of all progressively measurable sets in $\Omega \times [0,T]$ associated to the filtration $(\mf_t)_{t\geq0}$. To be more precise, $A\subset \Omega\times[0,T]$ belongs to the progressively measurable $\sigma$-field provided the stochastic process $(\omega,t)\mapsto \mathbb{I}_A(\omega,t)$ is $(\mf_t)$-progressively measurable. We denote by $L^1_{\rm{prog}}(\Omega\times[0,T])$  the Lebesgue space of functions that are measurable with respect to the $\sigma$-field  of $(\mf_t)$-progressively measurable sets in $\Omega \times [0,T]$ and we denote by  $\mu_{\rm{prog}}$ the measure $\p\otimes\mathfrak{L}_{[0,T]}$ restricted to the progressively measurable $\sigma$-field.
\begin{definition} \label{RDD3}
\index{adapted!random distribution}
Let $\bfU$ be a random distribution in the sense of Definition \ref{RDD1}.
\begin{enumerate}
\item[(a)]
We say that $\bfU$ is \emph{adapted}\index{random distribution!adapted} to $( \mathfrak{F}_t )$
if $\left< \bfU, \bfphi \right>$ is $(\mathfrak{F}_t)$-measurable for any $\bfphi \in C^\infty_c(Q_t)$.
\item[(b)] We say that $\bfU$ is \emph{$(\mf_t)$-progressively measurable} if $\left< \bfU, \bfphi \right>\in L^1_{\rm{prog}}(\Omega\times[0,T])$ for any  $\bfphi \in C^\infty_c(Q_T)$.
\end{enumerate}
%\item
%The family of $\sigma$-fields $ \left( \sigma_t[\vc{U}] \right)_{t \geq 0}$ given as\index{random distribution!history of}
%\[
%\sigma_t [\vc{U}] := \cap_{s > t} \sigma \left( \cup_{\bfphi \in \DC(Q_s)}
%\left\{ \left<\vc{U}, \bfphi \right> < 1 \right\}  \cup \{N\in\mf,\p(N)=0\}\right)
%\]
%is called \emph{history} of $\vc{U}$.
\end{definition}
The above concept is convenient when dealing with general distributions. It coincides with the standard concept of progressive measurability 
as long as the distribution defines a stochastic process, see \cite[Chapter 2, Lemma 2.2.18]{BFHbook}. Also, if a random distribution is $(\mf_t)$-adapted, there is a modification which is $(\mf_t)$-progressively measureable, cf. \cite[Chapter 2, Lemma 2.2.18]{BFHbook}, as in the classical situation.
\db{The family of $\sigma$-fields $ \left( \sigma_t[\bfU] \right)_{t \geq 0}$ given as\index{random distribution!history of}
\begin{align}\label{eq:2402a}
\sigma_t [\bfU] := \bigcap_{s > t} \sigma \bigg( \bigcup_{\bfvarphi \in C_c^\infty(Q_s )}
\left\{ \left<\bfU, \bfvarphi \right> < 1 \right\}  \cup \{N\in\mf,\p(N)=0\}\bigg)
\end{align}
is called \emph{history} of $\bfU$. Clearly, any random distribution is adapted to its history.}

\subsection{Stochastic analysis}
\label{subsec:prob}
Let $(\Omega,\mf,(\mf_t)_{t\geq0},\prst)$ be a complete stochastic basis with a Borel probability measure $\prst$ and a right-continuous filtration
$(\mf_t)$. We refer the reader to \cite{daprato} for more details on the following elements of stochastic calculus in infinite dimensions.
Let $\mathfrak U$ be a separable Hilbert space and let $(\bfe_k)_{k\in\N}$ be an orthonormal basis of $\mathfrak U$. We denote by $L_2(\mathfrak U,L^2(\mt))$ the set of Hilbert-Schmidt operators from $\mathfrak U$ to $L^2(\mt)$.  
Throughout the paper we consider a cylindrical Wiener process
$W=(W_t)_{t\geq0}$ which has the form
\begin{align}\label{eq:W}
W(\sigma)=\sum_{k\in\N}\beta_k(\sigma)\bfe_k
\end{align}
with a sequence $(\beta_k)$ of independent real valued Brownian motions on $(\Omega,\mf,(\mf_t)_{t\geq0},\prst)$. The stochastic integral
\begin{align*}
\int_0^t \psi\,\dd W,\quad \psi\in L^2(\Omega,\F,\p;L^2(0,T;L_2(\mathfrak U,L^2(\mt)))),
\end{align*}
where $\psi$ is $(\F_t)$-progressively measurable,
defines a $\p$-almost surely continuous $L^2(\mt)$ valued $(\F_t)$-martingale. Moreover, we can multiply with test-functions since
 \begin{align*}
\bigg\langle\int_0^t \psi\,\dd W,\bfphi\bigg\rangle_{L^2(\mt)}=\sum_{k=1}^\infty \int_0^t\langle\psi( \bfe_k),\bfphi\rangle_{L^2(\mt)}\,\dd\beta_k,\quad \bfphi\in L^2(\mt),
\end{align*}
is well-defined (the series converges in $L^2(\Omega,\F,\p; C[0,T])$).\\
Define further $\mathfrak U_0\supset \mathfrak U$  as
\begin{align}\label{eq:U0}
\mathfrak U_0:=\left\{\bfe=\sum_k \alpha_k\bfe_k\in \mathfrak U:\,\,\sum_k \frac{\alpha_k^2}{k^2}<\infty\right\},
\end{align}
thus the embedding $\mathfrak U\hookrightarrow \mathfrak U_0$ is Hilbert-Schmidt and trajectories of $W$ are $\p$-a.s. continuous with values in in $\mathfrak U_0$.\\
%In order to show regularity of solutions we suppose the following linear growth assumptions on $\Phi$: For each $\bfz\in L^2(G)$ there is a mapping $\Phi(\bfz):U\rightarrow L^{2}(G)^D$ defined by $\Phi(\bfz)\bfe_k=g_k(\cdot,\bfz(\cdot))$. In particular, we suppose
%that $g_k\in C^1(G\times\R^D)$ and the following conditions
%\begin{align}\label{eq:phi}\begin{aligned}
%\sum_{k\in\N}|g_k(x,\bfxi)|^2 \leq c(1+|\bfxi|^2)&,\quad
%\sum_{k\in\N}|\nabla_{\bfxi} g_k(x,\bfxi)|^2\leq c,\quad\bfxi\in\R^D,\\
%\sum_{k\in\N}|\nabla_x g_k(x,\bfxi)|^2 &\leq c(1+|\bfxi|^2).
%\end{aligned}
%\end{align}
The following infinite dimensional It\^{o}-formula is a variant of \cite[Lemma 3.1]{BFH}.
\begin{lemma} \label{lem}
Let $\StoB$ be a stochastic basis and let be $W$ a cylindrical $(\mathfrak F_t)$-Wiener process.
Let $\bfw^1,\bfw^2$ be $(\mathfrak F_t)$-progressively measurable satisfying
$\bfw^1\in C_w([0,T];L^2_{\Div}(\mt))$, $\bfw^2\in C([0,T];L^2_{\Div}(\mt))$ and $\bfw^2\in L^1(0,T;C^1(\mt))$ a.s. such that
\begin{align*}
\bfw^1,\bfw^2\in L^2_{w^*}(\Omega;L^\infty(0,T;L^2(\mt))).
\end{align*}
 Suppose that there are
\begin{align*}
\lambda_t\in L^1_{w^*}(\Omega;L^\infty_{w^*}(0,T&;\mathscr M^+(\mt))),\quad\Phi^1\in L^2(\Omega;L^2(0,T;L_2(\mathfrak U;L^2(\mt)))),\\
\bfH^1&\in L^1_{w^*}(\Omega;L^\infty(0,T;L^1(\mt))),
\end{align*}
\db{as well as a random distribution $\bfG^1$ such that $\bfG^1\in L^\infty(Q_T,\lambda_t\otimes\mathcal L^1)$ $\p$-a.s. and
\begin{align*}
\E\bigg[\inf_{\lambda_t\otimes\mathcal L^1(\mathscr N)=0}\|\bfG^1\|_{L^\infty(Q_T\setminus\mathscr N)}\bigg]<\infty.
\end{align*}}
We further assume that $\lambda_t$, $\Phi^1$, $\bfH^1$ and $\bfG^1$
are progressively $(\mathfrak F_t)$-measurable and that
\begin{align} \label{rel1}
\begin{aligned}
\int_{\mt}\bfw^{1}(t)\cdot\bfvarphi\dx&=\int_{\mt}\bfw^{1}(0)\cdot\bfvarphi\dx+\int_0^t\int_{\mt}\bfH^1:\nabla\bfvarphi\dxs\\&+\int_0^t\int_{\mt}\bfG^{1}:\nabla\bfvarphi\,\dd\lambda_\sigma\,\dd\sigma
+\int_{\mt}\bfvarphi\cdot\int_0^t\Phi^1\,\dd W\dx
\end{aligned}
\end{align}
for all $\bfvarphi\in C^\infty_{\Div}(\mt)$.\\
Suppose further that there are
\begin{align*}
\bfh^2\in L^{1}_{w^*}(\Omega;L^\infty(Q)),\quad\Phi^2\in L^2(\Omega;L^2(0,T;L_2(\mathfrak U;L^2(\mt)))),
\end{align*}
$(\mathfrak F_t)$-progressively measurable such that
such that
\begin{align} \label{rel2}
\begin{aligned}
\int_{\mt}\bfw^{2}(t)\cdot\bfvarphi\dx&=\int_{\mt}\bfw^{2}(0)\cdot\bfvarphi\dx+\int_0^t\int_{\mt}\bfh^2\cdot\bfvarphi\dxs\\
&+\int_{\mt}\bfvarphi\cdot\int_0^t\Phi^2\,\dd W\dx
\end{aligned}
\end{align}
for all $\varphi\in C^\infty_{\Div}(\mt)$.
Then we have for all $t\geq 0$ $\p$-a.s.
\begin{align} \nonumber
\int_{\mt} \bfw^1(t)\cdot\bfw^2(t)\dx
&=\int_{\mt} \bfw^1(0)\cdot\bfw^2(0)\dx+\int_0^t\int_{\mt}\bfH^1:\nabla\bfw^2\dxs\\&+\int_0^t\int_{\mt}\bfG^{1}:\nabla\bfw^2\,\dd\lambda_\sigma\,\dd\sigma
+\int_{\mt}\bfw^2\cdot\int_0^t\Phi^1\,\dd W\dx\nonumber\\
&+\int_0^t\int_{\mt}\bfh^2\cdot\bfw^1\dxs+\int_{\mt}\bfw^1\cdot\int_0^t\Phi^2\,\dd W\dx\nonumber\\
&+ \sum_{k\geq1}\int_0^t\int_{\mt}  \Phi^1 \bfe_k\cdot \Phi^2 \bfe_k\dx  \,{\rm d}t.
\label{result}
\end{align}
%where
%\begin{equation} \label{result1}
%\mathbb{M}_t = \sum_{k\geq1}\int_0^t\int_{\mt} \Big[  \bfw^1  \Phi^2 e_k  + \bfw^2 \Phi^1 e_k  \Big] \dx\,\Dif W_k.
%\end{equation}
\end{lemma}

\begin{proof}
In order to justify the application of It\^{o}'s formula to the process
$t\mapsto \int_{\mt} \bfw^1(t)\cdot\bfw^2(t)\dx$ we have to perform some regularisations in equation
\eqref{rel1} using mollification in space with parameter $\varrho>0$. For $\bfphi\in L^2_{\Div}(\mt)$ we have $\bfphi_\varrho\in C^\infty_{\Div}(\mt)$ and
\begin{align}\label{eq:0403a}
\begin{aligned}
\|\bfphi_\varrho\|_{W^{k,p}_x}&\leq\,c(\varrho)\|\bfphi\|_{L^2_x}\quad \forall k\in\N_0,\,\,p\in[1,\infty],\\
\|\bfphi_\varrho\|_{W^{k,p}_x}&\leq\,\|\bfphi\|_{W^{k,2}_x}\quad \forall k\in\N_0,\,\,p\in[1,\infty],
\end{aligned}
\end{align}
 provided $\bfphi\in L^{p}(\mt)$ or $\bfphi\in W^{k,p}(\mt)$ respectively.
Moreover,
\begin{align}\label{eq:0403b}
\begin{aligned}
\bfphi_\varrho&\rightarrow \bfphi\quad\text{in}\quad W^{k,p}(\mt)\quad\forall k\in\N_0,\,\,p\in[1,\infty),\\
\bfphi_\varrho&\rightarrow \bfphi\quad\text{in}\quad C^{k}(\mt)\quad\forall k\in\N_0,
\end{aligned}
\end{align}
as $\varrho\rightarrow0$ provided $\bfphi\in W^{k,p}(\mt)$ or $C^{k}(\mt)$ respectively.
 Finally, the operator $(\cdot)_\varrho$ commutes with derivatives. Inserting
$\bfphi_\varrho$ in \eqref{rel1} yields
\begin{align*}
\int_{\mt}\bfw^{1}_\varrho(t)\cdot\bfvarphi\dx&=\int_{\mt}\bfw^{1}_\varrho(0)\cdot\bfvarphi\dx+\int_0^t\int_{\mt}\bfH^1:\nabla(\bfvarphi)_\varrho\dxs\\&+\int_0^t\int_{\mt}\bfG^{1}:\nabla(\bfvarphi)_\varrho\,\dd\lambda_\sigma\,\dd\sigma
+\int_{\mt}\bfvarphi\cdot\int_0^t\Phi^1_\varrho\,\dd W\dx,
\end{align*}
where $\Phi^1_\varrho$ is given by $\Phi^1_\varrho \bfe_k=(\Phi^1 \bfe_k)_\varrho$ for $k\in\N$.
Using \eqref{eq:0403a} we have for fixed $\varrho>0$
\begin{align*}
\bigg|\int_0^t\int_{\mt}\bfH^1:\nabla(\bfvarphi)_\varrho\dxs\bigg|&\leq\sup_{0\leq t\leq T}\int_{\mt}|\bfH^1|\dx\int_0^T\|\nabla(\bfvarphi)_\varrho\|_{L^\infty_x} \ds\\
&\leq\,c(\varrho)\,\sup_{0\leq t\leq T}\int_{\mt}|\bfH^1|\dx\int_0^T\|\bfvarphi\|_{L^2_x} \ds
\end{align*}
$\p$-a.s. as well as
\begin{align*}
\bigg|\int_0^t\int_{\mt}\bfG^1:\nabla(\bfvarphi)_\varrho\,\dd\lambda_\sigma\ds\bigg|&\leq\sup_{Q_T}|\bfG^1|\int_0^T\int_{\mt}|\nabla(\bfvarphi)_\varrho| \,\dd\lambda_\sigma\ds\\
&\leq\sup_{Q_T}|\bfG^1|\sup_{0\leq \sigma\leq T}\lambda_\sigma(\mt)\int_0^T\|\nabla(\bfvarphi)_\varrho\|_{L^\infty_x}\ds\\
&\leq\,c(\varrho)\,\sup_{Q_T}|\bfG^1|\sup_{0\leq \sigma\leq T}\lambda_\sigma(\mt)\int_0^T\|\bfvarphi\|_{L^2_x}\ds.
\end{align*}
Hence the deterministic parts in the equation for $\bfw^1_\varrho$ are functionals on $L^2$. 
Consequently, we can apply It\^{o}'s formula on the Hilbert space $L^2_{\Div}(\mt)$ (see \cite[Thm. 4.17]{daprato})
to the process $t\mapsto \int_{\mt} \bfw^1_\varrho(t)\cdot\bfw^2(t)\dx$
to obtain
\begin{align*}
\int_{\mt} \bfw^1_\varrho(t)\cdot\bfw^2(t)\dx
&=\int_{\mt} \bfw^1_\varrho(0)\cdot\bfw^2(0)\dx+\int_0^t\int_{\mt}\bfH^1:(\nabla\bfw^2)_\varrho\dxs\\&+\int_0^t\int_{\mt}\bfG^{1}:(\nabla\bfw^2)_\varrho\,\dd\lambda_\sigma\,\dd\sigma
+\int_{\mt}\int_0^t\bfw^2\cdot\Phi^1_\varrho\,\dd W\dx\nonumber\\
&+\int_0^t\int_{\mt}\bfh^2\cdot\bfw^1_\varrho\dxs+\int_{\mt}\int_0^t\bfw^1_\varrho\cdot\Phi^2\,\dd W\dx\nonumber\\
&+ \sum_{k\geq1}\int_0^t\int_{\mt}  \Phi^1_\varrho \bfe_k\,\Phi^2 \bfe_k\dx \, {\rm d}t.
\end{align*}
Passing to the limit $\varrho\rightarrow0$ and using \eqref{eq:0403b} together with the assumptions on $\bfw^1$ and $\bfw^2$ we see that all terms converge to their corresponding counterparts and \eqref{result} follows.
\end{proof}

We conclude this section with a finite dimensional version of \cite[Chapter 2, Theorem 2.9.1]{BFHbook}. The proof of which follows along the same line (in fact, it is even simpler). 
\begin{proposition} \label{RDT3}
Let $U$  be a random distribution such that $U\in L^1_{\mathrm loc}([0,\infty))$ $\mathbb P$-a.s. Suppose that there is a bounded continuous function $b$ and a collection of random distributions $\mathbb G=(G_k)_{k=1}^\infty$ such that $\mathbb P$-a.s.
\begin{align*}
\sum_{k=1}^\infty|G_k|^2\in L^1_{\mathrm loc}([0,\infty)).
\end{align*}
Let $U_0$ be an $\mathfrak F_0$-measurable random variable and let
$W=(W_k)_{k=1}^\infty$ be a collection of real-valued independent Brownian motions. Suppose that
the filtration
\[
\mathfrak{F}_t = \sigma\Big(\sigma \big(U_0,\bfr_t{U},\bfr_t{W},\bfr_t\mathbb G\big)\Big),\ t \geq 0,
\]
is non-anticipative with respect to $W$.
Let $\tilde{U}_0$ be another random distribution and $\tilde{W}=(\tilde W_k)_{k=1}^\infty$ another stochastic process and random distributions and random distributions $\tilde{\mathbb G}=(\tilde G_k)_{k=1}^\infty$, such their joint laws coincide, namely,
\[
\mathcal{L}[ U_0,U, W,\mathbb G] = \mathcal{L}[ \tilde U_0,\tilde{U}, \tilde{W},\tilde{\mathbb G} ]\ \mbox{or}\ [ U_0,U, W,\mathbb G] \overset{d}{\sim} [ \tilde U_0,\tilde{U}, \tilde{W},\tilde{\mathbb G} ].
\]

Then $\tilde W$ is a collection of real-valued independent Wiener processes, the filtration
\[
\tilde{\mathfrak{F}}_t = \sigma\Big(\sigma \big(\tilde U_0,\bfr_t\tilde{U},\bfr_t\tilde{W},\bfr_t\tilde{\mathbb G}\big)\Big) ,\ t \geq 0,
\]
is non-anticipative with respect to $\tilde{W}$, $\tilde U_0$ is $\tilde{\mathfrak F}_0$-measurable, and
\begin{equation} \label{RD12}
\begin{split}
&\mathcal{L} \left[ \int_0^\infty \left[ \partial_t \psi U + b(U)\psi  \right] \dt
+ \int_0^\infty\sum_{k=1}^\infty \psi G_k\dd W_k + \psi(0) U_0 \right]\\
&=\mathcal{L}  \left[ \int_0^\infty \left[ \partial_t \psi \tilde U + b(\tilde U)\psi  \right] \dt
+ \int_0^\infty\sum_{k=1}^\infty \psi \tilde G_k\dd \tilde W_k + \psi(0) \tilde U_0 \right]
\end{split}
\end{equation}
for any deterministic $\psi \in C^\infty_c([0,\infty))$.

\end{proposition}

\subsection{Stochastic Navier--Stokes equations}
The Euler equations are linked via a vanishing viscosity limit to the Navier--Stokes equations.
The stochastic Navier--Stokes equations with viscosity $\mu>0$ read as
\begin{align}\label{eq:SNS}
\left\{\begin{array}{rc}
\dd\bfu=\mu\Delta\bfu\dt-(\nabla\bfu)\bfu\dt-\nabla\pi\dt+\Phi\dd W
& \mbox{in $Q$,}\\
\Div \bfu=0\qquad\qquad\qquad\qquad\qquad\,\,\,\,& \mbox{in $Q$,}\\
%\bfu=0\qquad\qquad\qquad\qquad\,\,\,\,\quad\quad& \mbox{ \,on $\partial \mathcal O$.}\\
%\bfu(0)=\bfu_0\,\qquad\qquad\qquad\qquad\qquad&\mbox{ \,in $\mathcal O$.}
\end{array}\right.
\end{align}
Here $W$ is a cylindrical Wiener process as introduced in the previous subsection. In the following we give a rigorous definition of a solution to \eqref{eq:SNS}.
\begin{definition}[Solution]\label{def:sol}
Let $\Lambda$ be a Borel probability measure on $L^2_{\Div}(\mt)$ and let $\Phi\in L_2(\mathfrak U;L^2(\mt))$. Then
$$\big((\Omega,\mf,(\mf_t),\prst),\bfu,W)$$
is called a finite energy weak martingale solution to \eqref{eq:SNS} with the initial data $\Lambda$ provided
\begin{enumerate}[(a)]
\item $(\Omega,\mf,(\mf_t),\prst)$ is a stochastic basis with a complete right-continuous filtration;
\item $W$ is an $(\mf_t)$-cylindrical Wiener process;
\item The velocity field $\bfu$ is $(\mf_t)$-adapted and satisfies $\prst$-a.s.
 $$\bfu\in C_{\mathrm loc}([0,\infty),W_{\Div}^{-2,2}(\mt))\cap C_{w,\mathrm loc}([0,\infty);L^{2}_{\Div}(\mt))\cap L^2_{\mathrm loc}(0,\infty;W^{1,2}_{\Div}(\mt));$$
\item $\Lambda=\prst\circ \big(\bfv(0) \big)^{-1}$;
%\item $\varPhi(\varrho,\varrho\bfv)\in L^2(\Omega\times[0,T],\mathcal{P},\dif\prst\otimes\dif t;L_2(\mathfrak{U};W^{-l,2}(\mt)))$ for some $l>\frac{3}{2}$,
\item For all
 $\bfvarphi\in C^\infty_{\Div}(\mt)$ and all $t\geq0$ there holds $\prst$-a.s.
\begin{align*}
\int_{\mt}\bfu(t)\cdot\bfvarphi\dx&=\int_{\mt}\bfu(0)\cdot\bfvarphi\dx+\int_0^t\int_{\mt}\bfu\otimes\bfu:\nabla\bfvarphi\dx\,\dif s\\
&-\mu\int_0^t\int_{\mt}\nabla\bfu:\nabla\bfvarphi\dx\,\dif s+\int_0^t\int_{\mt}\bfvarphi\cdot\varPhi\dx\,\dif W;
\end{align*}
\item 
%There exists a real valued $(\mf_t)$-martingale $M$ satisfying, for all $p\in[1,\infty)$,
%\begin{align*}
%\E \sup_{0\leq t\leq T}|M|^p\leq C \bigg(1+\E\bigg[\sup_{0\leq t\leq T}\int_{\mt}|\bfv(t)|^2\dx\bigg]^p\bigg),
%\end{align*}
The energy inequality holds in the sense that
\begin{align} \label{N3a}
\begin{aligned}
E_t+
\int_s^t\int_{\mt} &|\nabla \bfu|^2 \dx  \ds 
\\
& \leq E_s+\frac{1}{2} \int_s^t  \|\Phi\|_{L_2((\mathfrak U,L^2(\mt)))}^2   \ds
+ \int_s^t \int_{\mt} \bfu\cdot\Phi \, {\rm d}W
\end{aligned}
\end{align}
$\mathbb P$-a.s. for a.a. $s\geq0$ (including $s=0$) and all $\geq s$, where $E_t=\frac{1}{2}\int_ {\mt}|\bfu(t)|^2\dx $.
\end{enumerate}
\end{definition}
Definition \ref{def:sol} is standard in the theory of stochastic Navier--Stokes equations and can be found in a similar form, for instance, in \cite{FlGa} or \cite{FlaRom}. The energy inequality in (f) is in the spirit of \cite{FlaRom}, but slightly differs and is reminiscent of the recent result for compressible fluids from \cite{BFHbook}. Formerly, one can easily derive it by applying It\^{o}'s formula to the functional $t\mapsto\frac{1}{2}\int_ {\mt}|\bfu(t)|^2\dx$. It can be made rigorous on the Galerkin level (even with equality).
%. In the limit procedure once can argue, for instance, similarly to \cite[Lemma 4.3.16]{BFHbook}.
Consequently, the following existence theorem for \eqref{eq:SNS} holds.
\begin{theorem}\label{thm:NS}
Assume that we have
\begin{align*}
\int_{L^2_{\Div}(\mt)}\|\bfw\|^{p}_{L^2_x}\,\dd\Lambda(\bfw)<\infty
\end{align*}
for some $p>2$.
Then there is a martingale solution to \eqref{eq:SNS} in the sense of Definition \ref{def:sol}.
\end{theorem}

\section{Dissipative solutions}
\label{sec:DMS}

In this section we formalise the concept of dissipative solutions to the stochastic Euler equations
and prove their existence. The equations of interest read as 
\begin{align}\label{eq:Euler}
\left\{\begin{array}{rc}
\dd\bfu=-(\nabla\bfu)\bfu\dt-\nabla\pi\dt+\Phi\dd W
& \mbox{in $Q$,}\\
\Div \bfu=0\qquad\qquad\qquad\qquad\qquad\,\,\,\,& \mbox{in $Q$,}\\
%\bfu=0\qquad\qquad\qquad\qquad\,\,\,\,\quad\quad& \mbox{ \,on $\partial \mathcal O$.}\\
%\bfu(0)=\bfu_0\,\qquad\qquad\qquad\qquad\qquad&\mbox{ \,in $\mathcal O$,}
\end{array}\right.
\end{align}
Here $W$ is a cylindrical Wiener process as introduced in Section \ref{subsec:prob}.
Given an initial law $\Lambda$ on $L^2_{\Div}(\mt)$
a martingale solution to \eqref{eq:Euler} consists of a probability space
$(\Omega,\mf,(\mf_t),\prst)$ an $(\mf_t)$-cylindrical Wiener process and the random variables $(\bfu,\mathcal V)$.
The law $\mathcal L[\bfu(0),\bfu,\mathcal V,W]$ of $[\bfu(0),\bfu,\mathcal V,W]$ is a measure on the path path space
\begin{align}\label{eq:X}
\begin{aligned}
 \mathcal X&:= L^2_{\Div}(\mt)\times C_{\mathrm loc}([0,\infty);W^{-4,2}_{\Div}(\mt))\cap C_{w,\mathrm loc}([0,\infty);L^2_{\Div}(\mt))\\&\qquad\qquad\otimes Y_2^{\mathrm loc}(Q_\infty)\otimes C_{\mathrm loc}([0,\infty),\mathfrak U_0).
 \end{aligned}
 \end{align}
It is equipped with the $\sigma$-field
\begin{align}\label{eq:BX}\begin{aligned}
\mathscr B_{\mathcal X}&:=\mathscr B(L^2_{\Div}(\mt))\otimes \mathscr B_{\bfu}^{\mathrm loc}\otimes \mathscr B_Y^{\mathrm loc}\otimes \mathscr B(C_{\mathrm loc}([0,\infty),\mathfrak U_0)),\\
\mathscr B_{\bfu}^{\mathrm loc}&:=\sigma\big(\mathscr B(C_{\mathrm loc}([0,\infty);W^{-4,2}_{\Div}(\mt))\cap\mathscr B_\infty(C_{w,\mathrm loc}([0,\infty);L^2_{\Div}(\mt))\big),
\end{aligned}
\end{align}
where $\mathscr B_Y^{\mathrm loc}$ is defined in accordance with \eqref{eq:sigmaY}. For a Polish space $\mathscr Y$ we denote by $\mathscr{B}(\mathscr Y)$ its Borel $\sigma$-field and for a Banach space $X$ we denote by $\mathscr{B}_\infty(C_{w,\mathrm loc}([0,\infty);X))$ the $\sigma$-field generated by the mappings
$$C_{w,\mathrm loc}([0, \infty);X)\to X,\quad h\mapsto h(s),\quad s\geq0.$$
\begin{definition}[Dissipative Solution]\label{def:soleuler}
Let $\Lambda$ be a Borel probability measure on $L^2_{\Div}(\mt)$ and let $\Phi\in L_2(\mathfrak U;L^2(\mt))$. Then
$$\big((\Omega,\mf,(\mf_t),\prst),\bfu,\mathcal V,W)$$
is called a dissipative martingale solution to \eqref{eq:Euler} with the initial data $\Lambda$ provided
\begin{enumerate}[(a)]
\item $(\Omega,\mf,(\mf_t),\prst)$ is a stochastic basis with a complete right-continuous filtration;
\item $W$ is an $(\mf_t)$-cylindrical Wiener process;
\item The velocity field $\bfu$ is $(\mf_t)$-adapted and satisfies $\prst$-a.s.
 $$\bfu\in C_{\mathrm loc}([0,\infty),W^{-4,2}_{\Div}(\mt))\cap L^\infty_{\mathrm loc}(0,\infty;L^{2}_{\Div}(\mt));$$
  \item $\mathcal V=(\nu_{t,x},\nu^\infty_{t,x},\lambda)$ is $(\mf_t)$-adapted, we have $\mathcal V\in \bfY_2^{\mathrm loc}(Q_\infty,\R^3)$ $\prst$-a.s. and $\lambda=\lambda_t\otimes\mathcal L^1$ with $\lambda_t\in L^\infty_{w^*}(0,T;\mathscr M^+(\mt))$ $\prst$-a.s.;
  \item We have $\bfu(t,x)=\langle\nu_{t,x},\bfxi\rangle$ $\p$-a.s. for a.e. $(t,x)\in Q_\infty$;
\item $\Lambda=\prst\circ \big(\bfu(0) \big)^{-1}$ and $\mathcal L[\bfu(0),\mathcal V,\bfu,W]$ is a Radon measure on $(\mathcal X,\mathscr B_{\mathcal X})$;
%\item $\varPhi(\varrho,\varrho\bfv)\in L^2(\Omega\times[0,T],\mathcal{P},\dif\prst\otimes\dif t;L_2(\mathfrak{U};W^{-l,2}(\mt)))$ for some $l>\frac{3}{2}$,
\item For all
 $\bfvarphi\in C^\infty_{\Div}(\mt)$ and all $t\geq0$ there holds $\prst$-a.s.
 \begin{align}\label{eq:momentum}
 \begin{aligned}
\int_{\mt}\bfu(t)\cdot\bfvarphi\dx&=\int_{\mt}\bfu(0)\cdot\bfvarphi\dx+\int_0^t\int_{\mt}\big\langle\nu_{t,x},\bfxi\otimes\bfxi\big\rangle:\nabla\bfvarphi\dx\,\dif s\\&+\int_{(0,t)\times\mt}\big\langle\nu_{t,x}^\infty,\bfxi\otimes\bfxi\big\rangle:\nabla\bfvarphi\,\dif \lambda\
+\int_0^t\int_{\mt}\bfvarphi\cdot\varPhi\dx\,\dif W;
\end{aligned}
\end{align}
\item 
%There exists a real valued $(\mf_t)$-martingale $M$ satisfying, for all $p\in[1,\infty)$,
%\begin{align*}
%\E \sup_{0\leq t\leq T}|M|^p\leq C \bigg(1+\E\bigg[\sup_{0\leq t\leq T}\int_{\mt}|\bfv(t)|^2\dx\bigg]^p\bigg),
%\end{align*}
The energy inequality holds in the sense that
\begin{align} \label{N3}
\begin{aligned}
E_{t^+}\leq E_{s^-} +\frac{1}{2} \int_s^t  \|\Phi\|_{L_2((\mathfrak U,L^2(\mt)))}^2   \ds
+ \int_s^t  \int_{\mt} \bfu\cdot\Phi\dx \, {\rm d}W
\end{aligned}
\end{align}
 $\mathbb P$-a.s. for all $0\leq s<t$,
%for all $\phi\in C^\infty_c([0, \infty))$, $\phi \geq 0$ $\mathbb P$-a.s.,
%\begin{align}\label{energyeuler}
%\begin{aligned}
%\frac{1}{2}\int_ {\mt}\frac{1}{2}\big\langle\nu_{t,x},|\bfxi|^2\big\rangle\dx+\frac{1}{2}\lambda_t(\mt) &\leq \frac{1}{2}\int_{\mt}|\bfu(0)|^2\dx+\,\int_{\mt}\int_0^t\bfu\cdot\Phi\dd W\dx\\&+\frac{1}{2}\int_0^t\|\Phi\|_{L_2(\mathfrak U,L^2(\mt))}^2 \,\dd t,
%\end{aligned}
%\end{align}
%\begin{align} \label{N3}
%\begin{aligned}
%-  \int_0^\infty \partial_t \phi E_t\dt &- \phi(0) E_0
%\\
%& \leq \frac{1}{2} \int_0^\infty \phi  \|\Phi\|_{L_2((\mathfrak U,L^2(\mt)))}^2   \dt
%+ \int_0^\infty \phi  \int_{\mt} \bfu\cdot\Phi\dx \, {\rm d}W
%\end{aligned}
%\end{align}
%for all $\phi\in C^\infty_c([0, \infty))$, $\phi \geq 0$ $\mathbb P$-a.s.,
%\begin{equation} \label{N3}
%\begin{split}
%\frac{1}{n} \left[ \mathbb{E} \Big[ \mathbf 1_{\mathfrak{U}} {E}^n \Big] \right]_{t = \tau_1}^{t = \tau_2}
%&\leq \frac{2n-1}{2}\mathbb{E} \left[ \mathbf 1_{\mathfrak{U}} \int_{\tau_1}^{\tau_2} {E}^{n-1} \|\Phi\|_{L_2(\mathfrak U,L^2(\mt))}^2\dt                           \right]
%\end{split}
%\end{equation}
%holds for any $n = 0,1,\dots$, a.a. $\tau_2 \geq 0$ and a.a. $\tau_1$, $0 \leq \tau_1 \leq \tau_2$, including $\tau_1 = 0$, and any
%$\mathfrak{U} \in \sigma(\bfr_{\tau_1}\bfu,\bfr_{\tau_1}\mathcal V\big)$. Here 
where $E_t=\frac{1}{2}\int_ {\mt}\big\langle\nu_{t,x},|\bfxi|^2\big\rangle\dx+\frac{1}{2}\lambda_t(\mt) $ for $t\geq 0$ with $\lambda=\lambda_t\otimes \mathcal L^1$ and $E_{0^-}=\frac{1}{2}\int_{\mt}|\bfu(0)|^2\dx$. 
%In particular, we have
%\begin{align*}
%\E[E_t]\leq\,\frac{1}{2}\E\int_{\mt}|\bfu(0)|^2\dx+ \frac{1}{2}\E\int_0^t\|\Phi\|_{L_2(\mathfrak U,L^2(\mt))}^2\ds
%\end{align*}
%for all $t>0$.
%
%\begin{equation}\label{energy0}
%\begin{split}
%&\stred\bigg[\sup_{0\leq t\leq T}\int_ {\mt} \Big(\frac{1}{2}\big| \bfv(t)\big|^2+\nu |\nabla \bfv |^2\Big)\dx\,\dif s\bigg]^p\\
%& \leq \,C(p)\,\stred\bigg[\int_{\mt} \Big(\frac{1}{2} |\bfv(0)|^2 \dif x+1\bigg]^p.
%\end{split}
%\end{equation}
\end{enumerate}
\end{definition}
\begin{remark}\label{remark:energy}
Some remark concerning the energy inequality
\eqref{N3} are in order. At first glance it is not clear why the left- and right-sided limits
\begin{align*}
E_{t^+}=\lim_{\tau \searrow t}E_\tau,\quad E_{t^-}=\lim_{\tau \nearrow t}E_\tau
\end{align*}
exists in any time-point. Initially, we only show that
\begin{align*}
E_{t}\leq E_{s} +\frac{1}{2} \int_s^t  \|\Phi\|_{L_2((\mathfrak U,L^2(\mt)))}^2   \ds
+ \int_s^t  \int_{\mt} \bfu\cdot\Phi\dx \, {\rm d}W
\end{align*}
 $\mathbb P$-a.s. for a.a. $0<s<t$, see \eqref{eq:1607}. This, however, implies that the mapping
  \begin{align*}
t\mapsto  E_t  -\int_0^t \|\Phi\|_{L_2((\mathfrak U,L^2(\mt)))}^2   \ds
- \int_0^t  \int_{\mt} \bfu\cdot\Phi\dx \, {\rm d}W
\end{align*}
is non-increasing. Since it is also  pathwise bounded, left- and right-sided limits exist in all points. Furthermore,
$\int_0^\cdot \|\Phi\|_{L_2((\mathfrak U,L^2(\mt)))}^2   \dt$ and $\int_0^\cdot  \int_{\mt} \bfu\cdot\Phi\dx \, {\rm d}W$
are continuous such that left- and right-sided limits also exists for $E_t$. Finally, we obtain
$E_{t^+}\leq E_{t^-}$, such that there could be energetic sinks but no positive jumps in the energy.
\end{remark}

The main result of this section concerns the existence of a dissipative solution in the sense of Definition \ref{def:soleuler}.
\begin{theorem}\label{thm:main1}
Assume that we have
\begin{align*}
\int_{L^2_{\Div}(\mt)}\|\bfw\|^{p}_{L^2_x}\,\dd\Lambda(\bfw)<\infty
\end{align*}
for some $p>2$.
Then there is a dissipative martingale solution to \eqref{eq:Euler} in the sense of Definition \ref{def:soleuler}.
\end{theorem}
As a by-product of our proof, in which we approximate \eqref{eq:Euler} by a sequence of solutions to
\eqref{eq:SNS} with vanishing viscosity, we obtain the following result.
\begin{corollary}\label{cor:main}
Let $\Lambda$ be a given Borel probability measure on $L^2(\mt)$ such that
\begin{align*}
\int_{L^2_{\Div}(\mt)}\|\bfw\|^{p}_{L^2_x}\,\dd\Lambda(\bfw)<\infty
\end{align*}
for some $p>2$.
If
$\big((\Omega^\varepsilon,\mf^\varepsilon,(\mf^\varepsilon),\prst^\varepsilon),\bfu^\varepsilon,W^\varepsilon\big)$ is a finite energy weak martingale solution to \eqref{eq:SNS} in the sense of Definition \ref{def:sol} with the initial law $\Lambda$, then \db{there is a subsequence such that} 
\begin{align*}
\bfu_\varepsilon&\rightarrow\bfu\quad\text{in law on}\quad  C_{w,\mathrm loc}([0,\infty);L^{2}_{\Div}(\mt)),
\end{align*}
where $\bfu$ is a dissipative solution to \eqref{eq:Euler} in the sense of Definition \ref{def:soleuler} with the initial law $\Lambda$.
\end{corollary}
The rest of this section is dedicated to the proof of Theorem \ref{thm:main1} which we split in several parts.

\subsection{A priori estimates}
For any $\varepsilon>0$ Theorem \ref{thm:NS} yields the existence of a martingale solution
$$\big((\Omega^\varepsilon,\mf^\varepsilon,(\mf^\varepsilon_t),\prst^\varepsilon),\bfu^\varepsilon,W^\varepsilon)$$
to \eqref{eq:SNS}. Without loss of generality we can assume that the probability space as well as the Wiener process
$W^\varepsilon$ do not depend on $\varepsilon$, that is the solution is given by
$$\big((\Omega,\mf,(\mf_t),\prst),\bfu^\varepsilon,W).$$
%From \eqref{N3a} we obtain
%\begin{align*}
%\E\bigg[\|\bfu_\ep(T)\|_{L^2_x}^{2n}&+\ep\int_0^T\|\bfu_\ep\|_{L^2_x}^{2n-2}\int_{\mt}|\nabla\bfu_\ep|^2\dxt\bigg]\\&\leq\,c(n)\, \E\int_0^T \|\bfu_\ep\|_{L^2_x}^{2n-2} \|\Phi\|_{L_2(\mathfrak U,L^2(\mt))}^2\dt+c(n)\bigg[\int_{\mt}|\bfu_\ep(0)|^2\dx\bigg]^n
%\end{align*}
%for all $n\in\N$ and and all $T>0$. By our assumption on the initial law the last term can be bounded by constant $c(n,\Lambda)$. On the other hand, for arbitrary $\delta>0$
%\begin{align*}
%\E\int_0^T \|\bfu_\ep\|_{L^2_x}^{2n-2} \|\Phi\|_{L_2(\mathfrak U,L^2(\mt))}^2\dt&\leq\, \E\int_0^T\big(\|\bfu_\ep\|_{L^2_x}^{2n}+1\big)\|\Phi\|_{L_2(\mathfrak U,L^2(\mt))}^2\dt
%\end{align*}
%we conclude by Gronwall's lemma
%\begin{align}\label{eq:apriori1'}
%\E\bigg[\|\bfu_\ep(t)\|_{L^2_x}^{2n}+\ep\int_0^T\|\bfu_\ep\|_{L^2_x}^{2n-2}\int_{\mt}|\nabla\bfu_\ep|^2\dxt\bigg]\leq\,c(n,T,\Lambda,\Phi)
%\end{align}
%uniformly in $\varepsilon$.
%We can use Proposition \eqref{cor:B3FR} to estimate for any $R>0$
%\begin{align}\label{eq:apriori1}
%\begin{aligned}
%R\,\p\bigg(\sup_{0<t<T}&\|\bfu_\ep(t)\|_{L^2_x}^{2n}+\ep\int_0^T\|\bfu_\ep\|_{L^2_x}^{2n-2}\int_{\mt}|\nabla\bfu^\ep|^2\dxt\geq R\bigg)\\&\leq \,2\E\bigg[\|\bfu_\ep(0)\|_{L^2_x}^{2n}+\|\bfu_\ep(T)\|_{L^2_x}^{2n}+\ep\int_0^T\|\bfu_\ep\|_{L^2_x}^{2n-2}\int_{\mt}|\nabla\bfu^\ep|^2\dxt\bigg]
%\end{aligned}
%\end{align}
%which is bounded
%by \eqref{eq:apriori1'}.
From \eqref{N3a} we obtain for any $T>0$ (choosing $\phi=\mathbb I_{(0,t)}$ and $s=0$, taking the supremum with respect to $t$, the power $p$ and applying expectations)
\begin{align*}
\E\bigg[\sup_{0<t<T}&\int_{\mt}|\bfu^\ep|^2\dx+\ep\int_0^T\int_{\mt}|\nabla\bfu^\ep|^2\dxt\bigg]^{p}\\&\leq\,c(p)\int_{L^2_{\Div}(\mt)}\|\bfw\|^{p}_{L^2_x}\,\dd\Lambda(\bfw)+c(p)\,\E\bigg[\sup_{0<t<T}\int_{\mt}\int_0^t\bfu^\varepsilon\cdot\varPhi\,\dif W\dx\bigg]^p.
\end{align*}
By Burkholer-Davis-Gundi inequality we obtain
\begin{align*}
E\bigg[\sup_{0<t<T}\int_{\mt}\int_0^t\bfu^\varepsilon\cdot\Phi \,\dif W\dx\bigg]^p&\leq\,c(p)\E\bigg[\int_0^T\sum_{k\geq1}\bigg(\int_{\mt}\Phi \bfe_k\cdot\bfu^\varepsilon\dx\bigg)^2\bigg]^{\frac{p}{2}}\\
&\leq\,c(p)\E\bigg[\|\Phi\|^2_{L_2(\mathfrak U;L^2(\mt))}\int_0^T\int_{\mt}|\bfu^\varepsilon|^2\dxt\bigg]^{\frac{p}{2}}\\
&\leq\,c(p,\Phi,T)\int_0^T\E\bigg[\int_{\mt}|\bfu^\varepsilon|^2\dx\bigg]^{\frac{p}{2}}\dt.
\end{align*}
By Gronwall's lemma we conclude
\begin{align}\label{eq:apriori1}
\E\bigg[\sup_{0<t<T}\int_{\mt}|\bfu^\ep|^2\dx+\ep\int_0^T\int_{\mt}|\nabla\bfu^\ep|^2\dxt\bigg]^{p}\leq\,c(p,\Lambda,\Phi,T)
\end{align}
uniformly in $\varepsilon$.
We have to pass to the limit in the nonlinear convective term which requires some compactness arguments.
We write the momentum equation as 
\begin{align*}
\int_{\mt}\bfu^\ep(t)\cdot\bfphi\dx
&=\int_{\mt}\bfu^\ep(0)\cdot\bfphi\dx+\int_0^t\int_{\mt} \bfH^\ep:\nabla\mathcal\bfphi\dxs+\int_0^t\int_{\mt}\bfphi\cdot\Phi\dx\,\dd W,\\
\bfH^\ep&:=-\ep\nabla\bfu^\ep+\bfu^\ep\otimes\bfu^\ep,
\end{align*}
for all $\bfphi\in C^\infty_{\Div}(\mt)$. From the a priori estimates in \ref{eq:apriori1} we obtain
\begin{align}
\label{eq:p0}\bfH^\ep\in L^{1}(\Omega;L^2(0,T;L^1(\mt))\hookrightarrow L^{1}(\Omega;L^2(0,T;W^{-2,2}(\mt))
\end{align}
uniformly in $\varepsilon$.
Let us consider the functional
$$\mathscr H^\ep(t,\bfphi):=\int_0^t\int_{\mt} \bfH^\ep:\nabla\bfphi\dxs,\quad \bfphi\in C^\infty_{\Div}(\mt),$$
which is the deterministic part of the equation.
Then we deduce from \eqref{eq:p0} the estimate
\begin{align*}
\E\bigg[\big\|\mathscr H^\ep\big\|_{W^{1,2}(0,T;W_{\Div}^{-3,2}(\mt))}\bigg]\leq c(T).
\end{align*}
For the stochastic term we have
\begin{align*}
\E\bigg[&\Big\|\int_0^{\cdot}\Phi\,\dd W\Big\|^p_{C^{\alpha}([0,T]; L^2(\mt))}\bigg]\leq \,c\,\E\bigg[\int_0^T\|\Phi\|_{L_2(\mathfrak U,L^2(\mt))}^p\dt\bigg]=c(p,\Phi,T)
\end{align*}
for all $\alpha(1/p,1/2)$ and $p>2$.
Combining the two previous estimates and using the embeddings $W^{1,2}_t\hookrightarrow C^{1/2}_t$
and $L^2_x\hookrightarrow W^{-3,2}_{x}$ shows
\begin{align}\label{eq:vt2}
\E\Big[\|\bfu^\ep\|_{C^{\alpha}([0,T]; W_{\Div}^{-3,2}(\mt))}\Big]\leq c(T)
\end{align}
for all $\alpha<\frac{1}{2}$.

\subsection{Compactness}
We aim at proving tightness of the sequence of approximate solutions using
the compact embeddings
\begin{align}\label{eq:comp}
\begin{aligned}
C^{\alpha}([0,T];W^{-3,2}_{\Div}(\mt))&\hookrightarrow \hookrightarrow C([0,T];W^{-4,2}_{\Div}(\mt)),\\
C^{\alpha}([0,T];W^{-3,2}_{\Div}(\mt))\cap L^\infty(0,T;L^2_{\Div}(\mt))&\hookrightarrow \hookrightarrow C_w([0,T];L^2_{\Div}(\mt)).
\end{aligned}
\end{align}
For $T>0$ we consider the path space
$$ \mathcal X_T:= L^2_{\Div}(\mt)\times C([0,T];W^{-4,2}_{\Div}(\mt))\cap C_w([0,T];L^2_{\Div}(\mt))\otimes Y_2(Q_T,\R^3)\otimes C([0,T],\mathfrak U_0).$$
% as well as the extended paths space 
% $$ \mathcal X^E:=\mathcal X_E\times L^\infty_{w^*,\mathrm loc}([0,T))\times C_{\mathrm loc}([0,T\infty)),\quad  \mathcal X_T^E:=\mathcal X_E\times L^\infty_{w^*}(0,T)\times C([0,T]),$$
% which also include the energy $E^{\eps_m}_t=\frac{1}{2}\int_{\mt}|\bfu^{\eps_m}|^2\dx$ and the stochastic integral
%\begin{align*}
%\mathscr M^{\eps_m}_t=\int_0^t\int_{\mt}\bfu^{\eps_m}\cdot\Phi\,\dd  W
% \end{align*} 
% appearing in the energy inequality.\\
Clearly, tightness of $\mathcal{L}[\bfu_0,\bfr_T\bfu^{\eps_m},\bfr_T{\mathcal V}^{\eps_m},\bfr_TW]$ on $\mathcal X_E$ for any $T>0$
implies tightness of $\mathcal{L}[\bfu_0,\bfu^{\eps_m},{\mathcal V}^{\eps_m},W]$ on $\mathcal X$.
Here $\bfr_T$ is the restriction operator which restricts measurable functions (or space-time distributions)
defined on $(0,\infty)$ to $(0,T)$. It acts on various path spaces.
%So we need to show tightness of the measure
%\begin{align*}
%\bfnu^N:=\nu_{\bfv^N}\otimes \nu_{\bfW}\otimes\Lambda_{0}\otimes\Lambda_{\bff}.
%\end{align*}
We fix $T>0$ and consider the ball $\mathcal B_R$ in the space $$C^{\alpha}([0,T];W^{-3,2}_{\Div}(\mt))\cap L^\infty(0,T;L^2_{\Div}(\mt)).$$
We obtain for its complement by \eqref{eq:apriori1} and \eqref{eq:vt2}
\begin{align*}
\nu_{\bfr_T\bfu^\ep}&(\mathcal B_R^C)=\p\Big(\|\bfr_T\bfu^\ep\|_{C^{\alpha}_tW^{-3,2}_x}+\|\bfr_T\bfu^\ep\|_{L^\infty_tL_x^{2}}\geq R\Big)\leq  \frac{c}{R}.
\end{align*}
So, for any fixed $\eta>0$, we find $R(\eta)$ with
\begin{align*}
\mathcal L[\bfr_T\bfu_\ep](\mathcal B_{R(\eta)})&\geq 1-\eta,
\end{align*}
i.e. $\mathcal L[\bfr_T\bfu^\ep]$ is tight. 
Now we set $\mathcal V^\ep=(\delta_{\bfu^\ep},0,0)\in Y_2^{\mathrm loc}(Q_\infty)$ as the generalised Young measure associated to $\bfu^\ep$.
Similarly to the above we have
\begin{align*}
\mathcal L[\bfr_T\bfu^\ep](\mathcal B_{R(\eta)})&\geq 1-\eta,
\end{align*}
for some $R=R(\eta)$, where $\mathcal B_{R(\eta)}$ is now the ball in $L^\infty(0,T;L^2(\mt))$.
Recalling \eqref{eq:L2Y'} 
%we use the compact embedding
%\begin{align*}
%L^2(Q_T)\hookrightarrow \big(Y_2(Q_T),\rightharpoonup^\ast\big)
%\end{align*}
%to 
we conclude tightness of $\mathcal L[\bfr_T\mathcal V^\ep]$.\\
%Tightness of $\mathcal L[\bfr_T E^{\eps_m}]$ follows from \eqref{eq:apriori1}, whereas
%tightness of $\mathcal L[\bfr_T\mathscr M^{\eps_m}]$ is a consequence of the estimate
%(where $\alpha\in(\frac{1}{p},\frac{1}{2})$)
%\begin{align*}
%\E\bigg[\|\mathscr M^{\eps_m}\|^p_{C^\alpha([0,T])}\bigg]&\leq\,c\,\E\bigg[\int_0^T\bigg\|\int_{\mt}\bfu^{\eps_m}\cdot\Phi\dx\bigg\|^p_{L_2(\mathfrak U,\R)}\dt\bigg]\\
%&\leq\,c(T,\Phi)\bigg[\sup_{0< t< T}\int_{\mt}|\bfu^{\eps_m}|^2\dx\bigg]^p,
%\end{align*}
%\eqref{eq:apriori1}, and Arcela-Ascoli's theorem.
Since also the laws $\mathcal L[\bfr_TW]$ and $\mathcal L[\bfu_0]$ are tight, as
being a Radon measures on the Polish spaces $C([0,T],\mathfrak U_0)$ and $L^2_{\Div}(\mt)$, we can conclude that
$\mathcal L[\bfu_0,\bfr_T\bfu^\ep,\bfr_T\mathcal V^\ep,\bfr_T W]$ is tight on $\mathcal X_T$.
Since $T$ was arbitrary we conclude that $\mathcal L[\bfu_0,\bfu^\ep,\mathcal V^\ep, W]$ is tight on $\mathcal X$.
Now we use Jakubowski's version of the Skorokhod
representation theorem, see \cite{jakubow},
to infer the following result.
 Let us remark that that $\mathcal{L}[\bfu_0,\bfu^{\eps_m},{\mathcal V}^{\eps_m},W]$ is a sequence of tight measures on $(\mathcal X,\mathscr B_{\mathcal X})$. Consequently, its weak* limit is tight as well and hence Radon.
%\begin{remark}
%If we replace \eqref{energy} by \eqref{N3a} we only get 
%\begin{align}\label{eq:apriori1'}
%\E\bigg[\|\bfu_\ep(t)\|_{L^2_x}^{2n}+\ep\int_0^T\|\bfu_\ep\|_{L^2_x}^{2n-2}\int_{\mt}|\nabla\bfu^\ep|^2\dxt\bigg]\leq\,c(n,\Lambda,\Phi)
%\end{align}
%for a.a. $t$. This is however enough (with $n=1$) to conclude \eqref{eq:p0} (with $p=1$) and \eqref{eq:vt2}. On the other hand we can use Proposition \ref{cor:B3FR} to estimate for any $R>0$
%\begin{align*}
%R\,\p\bigg(\sup_{0<t<T}&\|\bfu_\ep(t)\|_{L^2_x}^{2n}+\ep\int_0^T\|\bfu_\ep\|_{L^2_x}^{2n-2}\int_{\mt}|\nabla\bfu^\ep|^2\dxt\geq R\bigg)\\&\leq \,2\E\bigg[\|\bfu_\ep(0)\|_{L^2_x}^{2n}+\|\bfu_\ep(T)\|_{L^2_x}^{2n}+\ep\int_0^T\|\bfu_\ep\|_{L^2_x}^{2n-2}\int_{\mt}|\nabla\bfu^\ep|^2\dxt\bigg]
%\end{align*}
%which is bounded
%by \eqref{eq:apriori1'}. This gives again the required compactness.
%\end{remark}

\begin{proposition}\label{prop:skorokhod}
There exists a nullsequence $(\eps_m)_{m\in\N}$, a complete probability space $(\tilde\Omega,\tilde\mf,\tilde\prst)$ with $(\mathcal{X},\mathscr B_{\mathcal X})$-valued random variables $( \tilde{\bfu}^{\eps_m}_0,\tilde\bfu^{\eps_m},\tilde{\mathcal V}^{\eps_m},\tilde W^{\eps_m})$, $m\in\N$, and $({\tilde\bfu}_0,\tilde\bfu,\tilde{\mathcal{V}},\tilde W)$ such that
\begin{enumerate}
 \item[(c)] For all $m\in N$ the law of $( \tilde{\bfu}^{\eps_m}_0,\tilde\bfu^{\eps_m},\tilde{\mathcal V}^{\eps_m},\tilde W^{\eps_m})$ on $\mathcal{X}$ is given by\\ $\mathcal{L}[\bfu_0,\bfu^{\eps_m},{\mathcal V}^{\eps_m},W]$;
\item[(b)] The law of $(\tilde\bfu_0,\tilde\bfu,\tilde{\mathcal V},\tilde W)$ is a Radon measure on $(\mathcal X,\mathscr B_{\mathcal X})$;
 \item[(c)] $( \tilde{\bfu}^{\eps_m}_0,\tilde\bfu^{\eps_m},\tilde{\mathcal V}^{\eps_m},\tilde W^{\eps_m})$ converges $\,\tilde{\prst}$-almost surely to $( \tilde{\bfu}_0,\tilde\bfu,\tilde{\mathcal V},\tilde W)$ in the topology of $\mathcal{X}$, i.e.
\begin{align} \label{wWS116}
\begin{aligned}
\tilde \bfu^{\eps_m}_0 &\to \tilde\bfu_0 \quad \mbox{in}\quad L^2(\mt) \ \tilde\p\mbox{-a.s.}, \\
\tilde \bfu^{\eps_m} &\to \tilde\bfu \quad \mbox{in}\quad C_{\mathrm loc}([0,\infty);W_{\Div}^{-4,2}(\mt)) \ \tilde\p\mbox{-a.s.}, \\
\tilde \bfu^{\eps_m} &\to \tilde\bfu \quad \mbox{in}\quad C_{w,\mathrm loc}([0,\infty);L^{2}_{\Div}(\mt)) \ \tilde\p\mbox{-a.s.}, \\
\tilde{\mathcal{V}}^{\eps_m} &\rightharpoonup^\ast \tilde{\mathcal V} \quad \mbox{in}\quad Y_2^{\mathrm loc}(Q_\infty) \ \tilde\p\mbox{-a.s.}, \\
\tilde W^{\eps_m} &\to \tilde W \quad \mbox{in}\quad C_{\mathrm loc}([0,\infty); \mathfrak{U}_0 )\ \tilde\p\mbox{-a.s.}
%\tilde E^{\eps_m} &\to \tilde E \quad \mbox{in}\quad L^\infty_{w^*;\mathrm loc}([0,\infty))\ \tilde\p\mbox{-a.s.}\\
%\tilde{\mathscr M}^{\eps_m} &\to \tilde{\mathscr M} \quad \mbox{in}\quad C_{\mathrm loc}([0,\infty))\ \tilde\p\mbox{-a.s.}\\
\end{aligned}
\end{align}
\end{enumerate}
\end{proposition}
It is now easy to show that we have $\tilde\p$-a.s.
\begin{align}\label{eq:2011}
\tilde\bfu^{\eps_m}(t,x)=\langle \tilde\nu_{t,x}^{\eps_m},\bfxi\rangle,\quad \tilde\bfu(t,x)=\langle \tilde\nu_{t,x},\bfxi\rangle\quad\text{for a.a.}\quad (t,x)\in Q_\infty,
\end{align}
where $\tilde{\mathcal V}^{\eps_m}=(\tilde\nu_{t,x}^{\eps_m},\tilde\nu_{t,x}^{\infty,\ep_m},\tilde\lambda^{\eps_m})$ and $\tilde{\mathcal V}=(\tilde\nu_{t,x},\tilde\nu_{t,x}^{\infty},\tilde\lambda)$. Indeed, for $T>0$ and $\bfpsi\in C^\infty_c(Q_T)$ we consider the mapping
\begin{align*}
(\bfw,\mathcal V)\mapsto \int_{Q_T}\big(\bfw-\langle \nu_{t,x},\bfxi\rangle\big)\cdot\bfpsi\dxt
\end{align*}
which is continuous on the paths space. We obtain from Proposition \ref{prop:skorokhod}
\begin{align*}
\int_{Q_T}\big(\tilde\bfu^{\eps_m}-\langle \tilde\nu^{\eps_m}_{t,x},\bfxi\rangle\big)\cdot\bfpsi\dxt\sim^d \int_{Q_T}\big(\bfu^{\eps_m}-\langle \nu^{\eps_m}_{t,x},\bfxi\rangle\big)\cdot\bfpsi\dxt=0,
\end{align*}
which implies the first claim from \eqref{eq:2011} by arbitrariness of $\bfpsi$ and $T$. Using again Proposition \ref{prop:skorokhod} we can pass to the limit $m\rightarrow\infty$ and the second assertion follows. Similarly, for any $T>0$ we can consider for
$f\in \mathcal G_2(Q_T)$ and $\varphi\in C(\overline Q_T)$ arbitrary the mappings
\begin{align*}
(\bfw,\mathcal V)\mapsto\int_{Q_T}&\varphi\langle\nu_{t,x}-\delta_{\bfw(t,x)},f(\bfxi)\rangle\dxt+\int_{Q_T}\varphi\langle\nu_{t,x}^{\infty},f^\infty(\bfxi)\rangle\,\dd\lambda
\end{align*}
to show that
\begin{align}\label{eq:1412}
\tilde{\mathcal V}^{\eps_m}=(\tilde\nu_{t,x}^{\eps_m},\tilde\nu_{t,x}^{\infty,\ep_m},\tilde\lambda^{\eps_m})=(\delta_{\tilde\bfu^{\eps_m}(t,x)},0,0)\quad\text{for a.a.}\quad (t,x)\in Q_\infty.
\end{align}

Now we introduce the filtration on the new probability space, which ensures the correct measurabilities of the new random variables.
%We denote by $\bfr_t$ the operator of restriction to the interval $[0,t]$ acting on various path spaces. In particular, if $X$ stands for one of the path spaces $L^r_{\mathrm loc}(0,\infty;L^r(\mathcal{O}))$ or $C_{\mathrm loc}([0,\infty),\mathfrak U_0)$ and $t\in[0,T]$, we define
%\begin{align}\label{restr}
%\bfr_t:X\rightarrow X|_{[0,t]},\quad f\mapsto f|_{[0,t]}.
%\end{align}
%Clearly, $ \bfr_t$ is a continuous mapping.
Let $(\tilde{\mathfrak F}_t)_{t\geq0}$ and $(\tilde{\mathfrak F}_t^{\eps_m})_{t\geq0}$ be the $\tilde{\p}$-augmented canonical filtration of the variables $\big(\tilde{\bfu}_0,\tilde{\bfu},\tilde{\mathcal{V}},\tilde{W}\big)$ and $\big(\tilde{\bfu}_0^{\eps_m},\tilde{\bfu}^{\eps_m},\tilde{\mathcal V}^{\eps_m},\tilde{W}^{\eps_m}\big)$, respectively, that is
\begin{align*}
\tilde{\mathfrak F}_t&=\sigma\Big(\sigma \big(\tilde{\bfu}_0,\bfr_t\tilde{\bfu},\bfr_t\tilde{W}\big)\cup\sigma_t[\tilde{\mathcal V}]\cup\big\{\mathcal N\in\tilde{\mathfrak F};\;\tilde{\p}(\mathcal N)=0\big\}\Big),\quad t\geq0,\\
\tilde{\mathfrak F}^{\eps_m}_t&=\sigma\Big(\sigma\big(\tilde{\bfu}^{\eps_m}_0,\tilde{\bfu}^{\eps_m},\bfr_t\tilde{W}^{\eps_m})\cup\sigma_t[\tilde{\mathcal V}^{\eps_m}]\cup\big\{\mathcal N\in\tilde{\mathfrak F};\;\tilde{\p}(\mathcal N)=0\big\}\Big),\quad t\geq0.
\end{align*}
\db{Here $\sigma_t$ denotes the history of a random distribution as defined in \eqref{eq:2402a}, where generalised Young measures are identified as random distribution in the sense of \eqref{eq:2402b}.}
The definitions above guarantee that the processes are adapted and we can define stochastic integrals.

\subsection{Concerning the new probability space}
\label{subsec:new}
Now are going to show that the approximated equations also hold on the new probability
space. We use the elementary method from \cite{on1} which has already been generalized to different settings (see, for instance, \cite{BrHo,Ho1}). The key idea is to identify the quadratic variation of the corresponding martingale as well as its cross variation with the limit Wiener process obtained through compactness. First we notice that $\tilde W$ has the same law as $W$. As a consequence, there exists a collection of mutually independent real-valued $(\tilde{\mf}_t^{\eps_m})_{t\geq0}$-Wiener processes $(\tilde{\beta}^{\eps_m}_k)$ such that $\tilde{W}^N=\sum_{k}\tilde{\beta}^{\eps_m}_k e_k$. In particular, there exists a collection of mutually independent real-valued $(\tilde{\mf}_t)_{t\geq0}$-Wiener processes $(\tilde{\beta}_k)$ such that $\tilde{W}=\sum_{k}\tilde{\beta}_k e_k$.
Let us now define for all $t\in[0,T]$ and $\bfvarphi\in C^\infty_{\Div}(\mt)$ the functionals
\begin{equation*}
\begin{split}
\mathfrak M^{\eps_m}(\bfu_0,\bfu,\mathcal V)_t&=\int_{\mt}\big(\bfu(t)-\bfu_0\big)\cdot\bfvarphi\dx-\eps_m\int_0^t\int_{\mt}\bfu\cdot\Delta\bfvarphi\dx\,\dif s\\
&-\int_0^t\int_{\mt}\big\langle\nu_{t,x},\bfxi\otimes\bfxi\big\rangle:\nabla\bfvarphi\dx\,\dif s-\int_{(0,t)\times\mt}\big\langle\nu_{t,x}^\infty,\bfxi\otimes\bfxi\big\rangle:\nabla\bfvarphi\,\dif \lambda\\
\mathfrak N_t&=\sum_{k=1}^\infty\int_0^t\bigg(\int_{\mt} \Phi\bfe_k\cdot\bfvarphi\dx\bigg)^2\ds,\quad
\mathfrak N^k_t=\int_0^t\int_{\mt} \Phi\bfe_k\cdot\bfvarphi\dxs.
\end{split}
\end{equation*}
By $\mathfrak M(\bfu^{\eps_m}(0),\bfu^{\eps_m},\V)_{s,t}$ we denote the increment $\mathfrak M(\bfu^{\eps_m}(0),\bfu^{\eps_m},\V)_{t}-\mathfrak M(\bfu^{\eps_m}(0),\bfu^{\eps_m},\V)_{s}$ and similarly for $\mathfrak N_{s,t}$ and $\mathfrak N^k_{s,t}$. 
Note that the proof will be complete once we show that the process $\mathfrak M(\tilde\bfu_0^{\eps_m},\tilde\bfu^{\eps_m},\tilde\V^{\eps_m})$ is an $(\tilde{\mf}_t^{\eps_m})_{t\geq0}$-martingale and its quadratic and cross variations satisfy, respectively,
\begin{equation}\label{mart}
\begin{split}
\langle\langle \mathfrak M(\tilde{\bfu}_0^{\eps_m},\tilde\bfu^{\eps_m},\tilde{\V}^{\eps_m})\rangle\rangle&=\mathfrak N,\qquad\langle\langle \mathfrak M(\tilde{\bfu}_0^{\eps_m},\tilde\bfu^{\eps_m},\tilde{\V}^{\eps_m}),\tilde \beta_k\rangle\rangle=\mathfrak N^k.
\end{split}
\end{equation}
Indeed, in that case we have
\begin{align}\label{neu3108}
\Big\langle\Big\langle \mathfrak M(\tilde{\bfu}^{\eps_m}_0,\tilde\bfu^{\eps_m},\tilde{\V}^{\eps_m})-\int_0^{\cdot} \int_{\mt}\bfphi\cdot\Phi\dx\,\dd\tilde W^{\eps_m}\Big\rangle\Big\rangle=0,
\end{align}
which implies the desired equation on the new probability space.
Let us verify \eqref{mart}. To this end, we claim that with the above uniform estimates in hand, the mapping
$$(\bfu_0,\bfu,\V)\mapsto \mathfrak M(\bfu_0,\bfu,\V)_t$$
is well-defined and continuous on the path space.
Hence we have
\begin{align*}
\mathfrak M^{\eps_m}(\bfu^{\eps_m}(0),\bfu^{\eps_m},\V^{\eps_m})&\sim^d \mathfrak M^\eps(\tilde\bfu^{\eps_m}_0,\tilde\bfu^{\eps_m},\tilde{\V}^{\eps_m}).
\end{align*}
Let us  now fix times $s,t\in[0,T]$ such that $s<t$ and let
$$h:\mathcal X\big|_{[0,s]}\rightarrow [0,1]$$
be a continuous function.
Since
$$\mathfrak M(\bfu^{\eps_m}(0),\bfu^{\eps_m},\V^{\eps_m})_t=\int_0^t\int_{\mt}\bfphi\cdot\Phi\dx\,\dd W=\sum_{k=1}^\infty\int_0^t\int_{\mt} \Phi\bfe_k\cdot\bfvarphi\dx\,\dd\beta_k$$
is a square integrable $(\mf_t)_{t\geq0}$-martingale, we infer that
$$\big[\mathfrak M^{\eps_m}(\bfu^{\eps_m}(0),\bfu^{\eps_m},\V^{\eps_m})\big]^2-\mathfrak N,\quad \mathfrak M^{\eps_m}(\bfu^{\eps_m}(0),\bfu^{\eps_m},\V^{\eps_m})\beta_k-\mathfrak N^k,$$
are $(\mf_t)_{t\geq0}$-martingales.
Let $\bfr_s$ be the restriction of a function to the interval $[0,s]$. Then it follows from the equality of laws in Proposition \ref{prop:skorokhod} that
\begin{align}\label{exp11}
&\tilde{\E}\big[\,h\big(\tilde{\bfu}^{\eps_m}_0,\bfr_s\tilde\bfu^{\eps_m},\bfr_s\tilde\V^{\eps_m},\bfr_s\tilde{W}^{\eps_m}\big)\mathfrak M^{\eps_m}(\tilde{\bfu}^{\eps_m}_0,\tilde\bfu^{\eps_m},\bfr_s\tilde\V^{\eps_m})_{s,t}\big]\\&=\E \big[\,h\big({\bfu}^{\eps_m}(0),\bfr_s\bfu^{\eps_m},\bfr_s\V^{\eps_m},\bfr_sW\big)\mathfrak M^{\eps_m}({\bfu}^{\eps_m}(0),\bfu^{\eps_m},\V^{\eps_m})_{s,t}\big]=0,
\nonumber
\end{align}
\begin{align}\label{exp21}
&\tilde{\E}\bigg[\,h\big(\tilde{\bfu}^{\eps_m}_0,\bfr_s\tilde\bfu^{\eps_m},\bfr_s\tilde\V^{\eps_m},\bfr_s\tilde{W}^{\eps_m}\big)\Big([\mathfrak M^{\eps_m}(\tilde{\bfu}^{\eps_m}_0,\tilde\bfu^{\eps_m},\tilde\V^{\eps_m})^2]_{s,t}-\mathfrak N_{s,t}\Big)\bigg]\\
&=\E\bigg[\,h\big({\bfu}^{\eps_m}(0),\bfr_s\bfu^{\eps_m},\V^{\eps_m},\bfr_sW\big)\Big([\mathfrak M^{\eps_m}({\bfu}^{\eps_m}(0),\bfu^{\eps_m},\V^{\eps_m})^2]_{s,t}-\mathfrak N_{s,t}\Big)\bigg]=0,
\nonumber
\end{align}
\begin{align}\label{exp31}
&\tilde{\E}\bigg[\,h\big(\tilde{\bfu}^{\eps_m}_0,\bfr_s\tilde\bfu^{\eps_m},\bfr_s\tilde\V^{\eps_m},\bfr_s\tilde{W}^{\eps_m}\big)\Big([\mathfrak M^{\eps_n}(\tilde{\bfu}^{\eps_m}_0,\tilde\bfu^{\eps_m},\tilde\V^{\eps_n})\tilde{\beta}_k^{\eps_m}]_{s,t}-\mathfrak N^k_{s,t}\Big)\bigg]\\
&=\E\bigg[\,h\big({\bfu}^{\eps_m}(0),\bfr_s\bfu^{\eps_m},\bfr_s\V^{\eps_m},\bfr_sW\big)\Big([\mathfrak M^{\eps_m}({\bfu}^{\eps_m}(0),\bfu^{\eps_m},\V^{\eps_m})\beta_k]_{s,t}-\mathfrak N^k_{s,t}\Big)\bigg]=0.
\nonumber
\end{align}
So we have shown \eqref{mart} and hence (\ref{neu3108}). On account of the convergences
from Proposition \ref{prop:skorokhod} and the higher moments from \eqref{eq:apriori1} we can pass to the limit in \eqref{exp11}--\eqref{exp31} and obtain the momentum equation in the sense of \eqref{eq:momentum}.\\
Let us finally consider the energy inequality in the sense of \eqref{N3}, for which we introduce
the abbreviations
\begin{align*}
\mathscr M^{\eps_m}_t=\int_0^t\int_{\mt}\bfu^{\eps_m}\cdot\Phi\dx\,\dd  W,\quad \tilde{\mathscr M}^{\eps_m}_t=\int_0^t\int_{\mt}\tilde\bfu^{\eps_m}\cdot\Phi\dx\,\dd  \tilde W^{\eps_m},
\end{align*}
for the stochastic integrals.
For the Navier--Stokes equations (on the original probability space)
with $E_t^{\eps_m}=\frac{1}{2}\int_{\mt}|\bfu^{\eps_m}|^2\dx$ we have
 \begin{align*}
\begin{aligned}
E_t^{\eps_m} \leq E_s^{\eps_m}+\frac{1}{2} \int_s^t \|\Phi\|_{L_2((\mathfrak U,L^2(\mt)))}^2   \dt
+ \mathscr M^{\eps_m}_t-\mathscr M^{\eps_m}_s
\end{aligned}
\end{align*}
for a.a. $s$ (including $s=0$) and all $t\geq s$, cf. \eqref{N3a}. For a fixed $s$ this is equivalent to
\begin{align*}
-  \int_s^\infty \partial_t \phi E^{\eps_m}_t\dt &- \phi(s) E^{\eps_m}_s
\\
& \leq \frac{1}{2} \int_s^\infty \phi  \|\Phi\|_{L_2((\mathfrak U,L^2(\mt)))}^2   \dt
+ \int_s^\infty \phi  \int_{\mt} \bfu^{\eps_m}\cdot\Phi\dx \, {\rm d}W
\end{align*}
$\p$-a.s. for all $\varphi\in C^\infty_c([s,\infty))$. Due to Propositions \ref{RDT3}  and \ref{prop:skorokhod} this continues to hold on the new probability space and we obtain
\begin{align*}
\tilde E_t^{\eps_m} \leq \tilde E_s^{\eps_m}+\frac{1}{2} \int_s^t \|\Phi\|_{L_2((\mathfrak U,L^2(\mt)))}^2   \ds
+ \tilde{\mathscr M}^{\eps_m}_t-\tilde{\mathscr M}^{\eps_m}_s
\end{align*}
$\tilde{\p}$-a.s. for a.a. $s$ (including $s=0$) and all $t\geq s$. Averaging in $t$ and $s$ yields
%\begin{align*}
%\begin{aligned}
%\dashint_{t-\varrho}^tE_\sigma^\varepsilon\ds &\leq \dashint_{s-\varrho}^sE_\tau^\varepsilon\,\dd\tau+\frac{1}{2} \dashint_{s-\varrho}^s\dashint_{t-\varrho}^t\int_\tau^\sigma \|\Phi\|_{L_2((\mathfrak U,L^2(\mt)))}^2   \dd s\ds\,\dd\tau+ \dashint_{s-\varrho}^s\dashint_{t-\varrho}^t {\mathscr M}_s\,\dd\tau\ds
%\end{aligned}
%\end{align*}
%provided $s>0$ and $\varrho<\min\{s,t-s\}$ (the easier case $s=0$ will be treated at the end).
%All terms are now continuous on the path space (because of the additional time integrals).
%So, this carries over to the new probability space by Proposition \ref{prop:skorokhod} and we obtain
%we can pass to the limit in $\varepsilon$. We obtain $\p$-a.s.
\begin{align}\label{eq:1907}
\begin{aligned}
\dashint_{t-\varrho}^t\tilde E^{\eps_m}_r\,\dd r &\leq \dashint_{s-\varrho}^s\tilde E^{\eps_m}_\tau\,\dd\tau+\frac{1}{2} \dashint_{s-\varrho}^s\dashint_{t-\varrho}^t\int_\tau^r \|\Phi\|_{L_2((\mathfrak U,L^2(\mt)))}^2   \ds\,\dd r\,\dd\tau\\&+ \dashint_{s-\varrho}^s\dashint_{t-\varrho}^t (\tilde{\mathscr M}^{\eps_m}_r-\tilde{\mathscr M}^{\eps_m}_\tau)\,\dd r\,\dd\tau
\end{aligned}
\end{align}
provided $s>0$ and $\varrho<\min\{s,t-s\}$ (the easier case $s=0$ will be treated at the end).
We aim to pass to the limit first in $m$ and then in $\varrho$. The terms in \eqref{eq:1907} involving the energy are continuous on the path space due to the additional time integrals. Hence they converge $\tilde{\mathbb P}$-a.s. as $m\rightarrow\infty$ to the expected limits by Proposition \ref{prop:skorokhod}. In order to prove
that as $m\rightarrow\infty$ we have
\begin{align}\label{eq:1907b}
 \tilde{\mathscr M}^{\eps_m}\rightarrow \tilde{\mathscr M}:=\int_0^t\int_{\mt}\tilde\bfu\cdot\Phi\dx\,\dd  \tilde W\quad\text{in}\quad L^2_{\mathrm loc}([0,\infty))
 \end{align}
 in probability we aim to apply \cite[Lemma 2.1]{debussche1}. Hence we need to know in addition to \eqref{wWS116}$_5$ that
\begin{align}\label{eq:2406}
\int_{\mt} \tilde\bfu^{\varepsilon_m}\cdot\Phi\dx\rightarrow \int_{\mt} \tilde\bfu\cdot\Phi\dx\quad\text{in}\quad L^2_{\mathrm loc}([0,\infty);L_2(\mathfrak U;\mathbb R))
\end{align}
in probability. By \eqref{wWS116}$_3$ we have $\tilde{\mathbb P}$-a.s.
\begin{align*}
\int_{\mt} \tilde\bfu^{\varepsilon_m}(t)\cdot\Phi\dx\rightarrow \int_{\mt} \tilde\bfu(t)\cdot\Phi\dx\quad\text{in}\quad L_2(\mathfrak U;\mathbb R)
\end{align*}
for all $t\geq0$. Hence we also obtain convergence in $L^2(\tilde \Omega;L_2(\mathfrak U;\mathbb R))$ using the higher moments from \eqref{eq:apriori1}. Finally, we can use again
\eqref{eq:apriori1} to obtain \eqref{eq:2406} (in fact, we even have $L^2(\tilde\Omega)$-convergence).
In conclusion we can pass to the limit in \eqref{eq:1907} (first in $m$ and then in $\varrho$)
to obtain
  \begin{align}\label{eq:1607}
\tilde E_t \leq \tilde E_s+\frac{1}{2} \int_s^t \|\Phi\|_{L_2((\mathfrak U,L^2(\mt)))}^2   \dt
+ \tilde{\mathscr M}_t-\tilde{\mathscr M}_s
\end{align}
provided $t,s$ are Lebesgue points of $\tilde E_t=\frac{1}{2}\int_ {\mt}\big\langle\tilde\nu_{t,x},|\bfxi|^2\big\rangle\dx+\frac{1}{2}\tilde\lambda_t(\mt)$. \db{Here we also used that 
$\frac{1}{\varrho}\tilde{\E}\tilde\lambda((t-\varrho,t)\times\mt)$
stays bounded in $\varrho$ by \eqref{eq:1907}, which shows that $\tilde\lambda=\tilde\lambda_t\otimes\mathcal L^1$ with $\tilde\lambda_t\in L^\infty_{w^*}(0,T;\mathscr M^+(\mt))$ $\tilde\p$-a.s.}
Relation \eqref{eq:1607} implies that the function
  \begin{align*}
t\mapsto \tilde E_t  -\int_0^t \|\Phi\|_{L_2((\mathfrak U,L^2(\mt)))}^2   \ds
- \tilde{\mathscr M}_t
\end{align*}
is non-increasing. Since it is also  pathwise bounded (recall again \eqref{eq:apriori1}), left- and right-sided limits exist in all points. Furthermore,
$\int_0^\cdot \|\Phi\|_{L_2((\mathfrak U,L^2(\mt)))}^2   \ds$ and $\tilde {\mathscr M}$
are continuous such that left- and right-sided limits also exists for $\tilde E_t$. Approximating arbitrary
$t$ and $s$ by Lebesgue points and using \eqref{eq:1607} we have
  \begin{align}\label{eq:1607'}
\tilde E_{t^+} \leq \tilde E_{s^-}+\frac{1}{2} \int_s^t \|\Phi\|_{L_2((\mathfrak U,L^2(\mt)))}^2   \dt
+ \tilde{\mathscr M}_t-\tilde{\mathscr M}_s
\end{align}
$\tilde{\p}$-a.s. for all $t>s>0$.
If $s=0$ we argue similarly to \eqref{eq:1907} but without
the averaging in $s$. We obtain 
\begin{align*}
\begin{aligned}
\dashint_{t-\varrho}^t\tilde E_r^{\varepsilon_m}\,\dd r &\leq \tilde E_0^{\varepsilon_m}+\frac{1}{2} \dashint_{t-\varrho}^t\int_0^r \|\Phi\|_{L_2((\mathfrak U,L^2(\mt)))}^2   \ds\,\dd r+ \dashint_{t-\varrho}^t \tilde{\mathscr M}^{\eps_m}_r\,\dd r
\end{aligned}
\end{align*}
$\tilde{\p}$-a.s. provided $\varrho<t$. Since $E_0^{\varepsilon_m}=\frac{1}{2}\int_{\mt}|\bfu_{0}^{\varepsilon_m}|^2\dx$ we can argue again by Proposition \ref{prop:skorokhod} and \eqref{eq:1907b} to conclude
  \begin{align}\label{eq:1607c}
\tilde E_t \leq \tilde E_{0^-}+\frac{1}{2} \int_0^t \|\Phi\|_{L_2((\mathfrak U,L^2(\mt)))}^2   \ds
+ \tilde{\mathscr M}_t
\end{align}
$\tilde{\p}$-a.s. for Lebesgue points $t$, where $\tilde E_{0^-}=\frac{1}{2}\int_{\mt}|\tilde\bfu_{0}|^2\dx$. Finally, we also obtain
  \begin{align*}
\tilde E_{t^+} \leq \tilde E_{0^-}+\frac{1}{2} \int_0^t \|\Phi\|_{L_2((\mathfrak U,L^2(\mt)))}^2   \ds
+ \tilde{\mathscr M}_t
\end{align*}
$\tilde{\p}$-a.s. for all $t>0$. This, in combination with \eqref{eq:1607'},
finishes the proof of the energy inequality \eqref{N3}. The proof of Theorem \ref{thm:main1}
is hereby complete.

\section{Weak-strong uniqueness}
\label{sec:weakstrong}
In this section we compare the dissipative solution from Definition \ref{def:soleuler} with a strong solution. The results are reminiscent of those from \cite{BFH} on the compressible Navier--Stokes system. A strong solution to the stochastic Euler equations is known to exists at least in short time. A concept which we make precise in the following.  

\begin{definition}\label{def:strsol}

Let $(\Omega,\mf,(\mf_t),\prst)$ be a stochastic basis with a complete right-continuous filtration, let ${W}$ be an $\left( \mathfrak{F}_t \right)$-cylindrical Wiener process. A random variable
 $\bfu$ and a stopping time $\mathfrak{t}$ is called a (local) strong solution to system \eqref{eq:Euler} provided
\begin{enumerate}[(a)]
\item the process $t \mapsto \bfu (t\wedge \mathfrak t, \cdot) $ is $\left( \mathfrak{F}_t \right)$-adapted, $\bfu (t\wedge \mathfrak t, \cdot),\nabla\bfu (t\wedge \mathfrak t, \cdot) \in C_{\mathrm loc}([0,\infty)\times\mt)$ $\mathbb P$-a.s. and for all $T>0$
\[
\E\bigg[\sup_{0\leq t\leq T}\big(\|\nabla\bfu(\cdot\wedge \mathfrak t)\|_{L^\infty_x}+\|\bfu(\cdot\wedge \mathfrak t)\|_{L^\infty_x}\big) \bigg] < \infty;
\]
\item for all
 $\bfvarphi\in C^\infty_{\Div}(\mt)$ and all $t\geq0$ there holds $\prst$-a.s.
 \begin{align*}
\int_{\mt}\bfu(t\wedge \mathfrak t)\cdot\bfvarphi\dx&=\int_{\mt}\bfu(0)\cdot\bfvarphi\dx-\int_0^{t\wedge \mathfrak t}\int_{\mt}(\nabla\bfu)\bfu\cdot\bfvarphi\dx\,\dif s
\\&+\int_0^{t\wedge \mathfrak t}\int_{\mt}\bfvarphi\cdot\varPhi\dx\,\dif W;
\end{align*}
\item we have $\Div\bfu (\cdot\wedge \mathfrak t)=0$ $\mathbb P$-a.s.
\end{enumerate}
\end{definition}
\begin{remark}
A direct application of It\^{o}'s formula (in the Hilbert space version for $L^2_{\Div}(\mt)$) shows that strong solutions satisfy the energy equality
\begin{align}\label{eq:energy=}
\int_{\mt}|\bfu(t)|^2\dx=\int_{\mt}|\bfu(0)|^2\dx+2\int_0^{t}\int_{\mt}\bfu\cdot\varPhi\dx\,\dif W+\int_0^t\|\Phi\|_{L_2(\mathfrak U,L^2(\mt))}^2\dt
\end{align}
for all $t\in[0,\mathfrak t]$ $\p$-a.s.
\end{remark}
The existence of local-in-time strong solutions to \eqref{eq:Euler} (however, under slip boundary conditions and not in the periodic setting) in the sense of Definition \ref{def:strsol} was established in \cite[Theorem 4.3]{GHVic}
under certain assumptions imposed on the coefficient $\Phi$.

\subsection{Pathwise weak-strong uniqueness}
We begin with the case that the dissipative solution and the strong solution are defined on the same probability space. We have the following result concerning weak-strong uniqueness. 
\begin{theorem}\label{thm:uniq}
The pathwise weak-strong uniqueness holds true for the stochastic Euler equations \eqref{eq:Euler} in the following sense: let $$\big((\Omega,\mf,(\mf_t),\prst),\bfu,\mathcal V,W)$$ be a dissipative martingale solution to  \eqref{eq:Euler} in the sense of Definition \ref{def:soleuler} and let $\bfv$ and a stopping time 
$\mathfrak{t}$ be a strong solution of the same problem in the sense of Definition \ref{def:strsol} defined on the same stochastic basis with the same Wiener process and with the same initial data
(meaning $\bfv (0, \cdot) = \bfu (0, \cdot)$ $\mathbb{P}$-a.s.).
Then we have for a.a. $(t,x)$ that
$\bfu (t\wedge \mathfrak{t},x) = 
{\bf v} (t \wedge \mathfrak{t},x)$ and $(\nu_{t\wedge \mathfrak t,x},\nu_{t\wedge \mathfrak t,x}^\infty,\lambda)=(\delta_{\bfu(t\wedge \mathfrak t,x)},0,0)$ $\mathbb P$-a.s.
\end{theorem}
\begin{proof}
We start by introducing the stopping time
\[
\tau_L = \inf \Big\{ t \in (0,\mathfrak t) \ \big| \  \ \| \nabla\bfv (t, \cdot) \|_{L^\infty_x} > L \Big\},\quad L>0,
\]
and define $\tau_L=\mathfrak t$ if $\{\dots\}=\emptyset$.
Since $\E\big[\sup_{t\in[0,\mathfrak t]}\|\nabla\bfv(t)\|_{L^\infty_x}\big]<\infty$ by assumption (recall Definition \ref{def:strsol}) we have
\[
\mathbb{P} \left[ \tau_L < \mathfrak t \right]\leq \mathbb{P} \left[ \sup_{t\in[0,\mathfrak t]}\|\nabla\bfv(t)\|_{L^\infty_x}\geq L\right]\leq\frac{1}{L}\E\bigg[\sup_{t\in[0,\mathfrak t]}\|\nabla\bfv(t)\|_{L^\infty_x}\bigg]\rightarrow 0
\]
as $L\rightarrow\infty$ by Tschebyscheff's inequality.
Consequently, we have 
\begin{align}\label{eq:3003}
\tau_L\rightarrow \mathfrak t\quad  \text{in probability}.
\end{align}
Whence it is enough to show the claim in $(0,\tau_L)$ for a fixed $L$.
We consider the functional
\begin{align*}
F(t) =  \frac{1}{2}\int_ {\T}\big\langle \nu_{t,x},|\boldsymbol{\xi}-\bfv|^2\big\rangle\dx+\frac{1}{2}\lambda_t(\T)
\end{align*}
defined for a.a. $t<\mathfrak t$.
Noting that $\bfu=\langle \nu_{t,x},\boldsymbol{\xi} \rangle$ we can write
\begin{align*}
F(t) &= \frac{1}{2}\bigg(\int_ {\T} \langle \nu_{t,x},|\boldsymbol{\xi}|^2 \rangle +\lambda_t(\T)-2\int_ {\T}\bfu\cdot \bfv\dx + \int_{\T}|\bfv|^2  \dx\bigg),\\
&=E(t) + \frac{1}{2}\int_{\T}|\bfv|^2  \dx-2\int_ {\T}\bfu\cdot \bfv\dx.
\end{align*}
This definition can be extended to any $t<\mathfrak t$ by setting
\begin{align*}
F(t)
&=E(t^+) + \frac{1}{2}\int_{\T}|\bfv|^2  \dx-2\int_ {\T}\bfu\cdot \bfv\dx
\end{align*}
recalling that $\bfu$ and $\bfv(\cdot\wedge \mathfrak t)$ belong to $C_w([0,T];L^2(\mt))$.
%Applying It\^{o}'s lemma it is straightforward to see that $\bfv$ satisfies an energy equality, that is we have $\mathbb P$-a.s.
%\begin{align*}
%\frac{1}{2}\int_{\T}|\bfv(t)|^2\dx=\frac{1}{2}\int_{\T}|\bfv(0)|^2\dx+\int_{\T}\int_0^t\bfv\cdot\Phi\,\dd W\dx+\frac{1}{2}\int_0^t\|\phi\|_{L_2(\UU,L^2(\T))}^2
%\end{align*}
%for all $t$.
Taking the expectation of $F(t\wedge \tau_L)$ and using (\ref{N3}) and \eqref{eq:energy=} yields
\begin{align*}
   \E &[F(t\wedge \tau_L )] \\&= \E [E((t\wedge \tau_L)^+)] +\frac{1}{2}\E \int_{\T}|\bfv(t\wedge \tau_L )|^2 \dd x - \E \int_{\T}\bfu(t\wedge \tau_L)\cdot\bfv(t\wedge \tau_L) \dd x \\
   &\leq \E \left( \int_{\T}|\bfv(0)|^2\dx+\int_0^{t\wedge \tau_L}\|\Phi\|_{L_2(\UU,L^2(\T))}^2 \,\dd \sigma\right)- \E \int_{\T}\bfu(t\wedge \tau_L )\cdot \bfv (t\wedge \tau_L )\,\dd x,
\end{align*}
where we also used $\bfu(0)=\bfv(0)$.
Re-writing the last term using Lemma \ref{lem}, we infer that
\begin{align*}
  \mathfrak A& := \int_{\T}\bfu(t\wedge \tau_L )\cdot\bfv(t\wedge \tau_L ) \dd x \\
  &= \int_{\T}\bfu(0)\cdot\bfv(0) \dd x +\int_0^{t\wedge \tau_L } \int_{\T} \langle \nu_{t,x},\bxi \otimes \bxi \rangle: \nabla \bfv \dxs\\& + \int_{(0,t\wedge \tau_L )\times\T}\langle \nu^{\infty}_{t,x},\bfxi \otimes \bfxi\rangle : \nabla \bfv\, \dd \lambda + \int_0^{t\wedge \tau_L}\int_{\T} (\bfv+\bfu)\cdot \Phi\dx\, \dd W 
   \\&+ \int_0^{t\wedge \tau_L }\int_{\T} \mathrm{div}(\bfv\otimes \bfv)\cdot \bfu \dxs +\int_0^{t\wedge \tau_L}\|\Phi\|_{L_2(\UU,L^2(\T))}^2 \,\dd t.
\end{align*}
The stochastic term in $\mathfrak A$ vanish upon computing expectations. Using also $\bfu(0)=\bfv(0)$ we obtain
\begin{eqnarray}
\E[F(t\wedge \tau_L)] &\leq -\E(\mathfrak A_I + \mathfrak A_{II}+\mathfrak A_{III})
         \label{ESt_a}
\end{eqnarray}
with the remaining terms
\begin{align*}
\mathfrak A_I&=\int_0^{t\wedge \tau_L} \int_{\T} \langle \nu_{t,x},\bxi \otimes \bxi \rangle: \nabla \bfv \dxs,\\
\mathfrak A_{II}&=\int_{(0,t\wedge \tau_L)\times\T}\langle \nu^{\infty}_{t,x},\bfxi \otimes \bfxi\rangle : \nabla \bfv\, \dd \lambda,\\
\mathfrak A_{III}&=\int_0^{t\wedge \tau_L} \int_{\T} \mathrm{div}(\bfv\otimes \bfv)\cdot \bfu \dxs.
\end{align*} 
Using standard identities for the nonlinear term we can write
\begin{equation}
    \mathfrak A_I+\mathfrak A_{III}=\int_0^{t\wedge \tau_L} \int_{\T} \langle \nu_{t,x},(\bxi-\bfv)  \otimes (\bxi-\bfv) \rangle: \nabla \bfv \dxs,
    \label{D_form}
\end{equation}
such that
\begin{align*}
\E[F(t\wedge \tau_L)]\leq& -\E\int_0^{t\wedge \tau_L} \int_{\T} \langle \nu_{t,x},(\bxi-\bfv)  \otimes (\bxi-\bfv) \rangle: \nabla \bfv \dxs\\&-\E\int_{(0,t\wedge \tau_L)\times\T}\langle \nu^{\infty}_{t,x},\bfxi \otimes \bfxi\rangle:\nabla \bfv\, \dd\lambda\\
\leq&\,\E\int_0^{t\wedge \tau_L} \int_{\T} \langle \nu_{t,x},|\bxi - \bfv|^2 \rangle\,|\nabla \bfv| \dxs+\E\int_0^{t\wedge \tau_L}\int_{\T}|\nabla \bfv|\, \dd\lambda_\sigma\, \dd \sigma 
\\\leq&\, \E\int_0^{t\wedge \tau_L} F(\sigma)\| \nabla \bfv \|_{L^\infty_x} \dd \sigma\leq\,L\,\E\int_0^{t\wedge \tau_L} F(\sigma)\ds
\label{ESt_b}
\end{align*}
by definition of $\tau_L$.
 Finally, Gronwall's lemma implies that $\E [F(t\wedge \tau_L)]=0$ for a.e. t as required.
 Using \eqref{eq:3003} we obtain $F(t\wedge \mathfrak t)=0$ $\mathbb P$-a.s. This finally yields the claim by definition of $F$.
\end{proof}

\begin{remark}\label{rem:new}
Suppose that $\mathfrak t=T$ is deterministic. As can be seen from the proof, in this case the conclusion  of Theorem \ref{thm:uniq} can be slightly strengthened to $\bfu = 
{\bf v}$ and $\mathcal V=(\delta_{\bfu},0,0)$ $\mathbb P$-a.s., that is 
\begin{align*}
&\p\Big(\Big\{\bfu(t,x)=\bfv(t,x)\,\,\text{for a.a. }(t,x)\in Q_T\Big\}\Big)=1,\\
&\p\Big(\Big\{(\nu_{t,x},\nu_{t,x}^\infty,\lambda)=(\delta_{\bfu(t,x)},0,0)\,\,\text{for a.a. }(t,x)\in Q_T\Big\}\Big)=1.
\end{align*}
\end{remark}

\subsection{Weak-strong uniqueness in law}

In this subsection we are finally concerned with the case that the dissipative solution and the strong solution are defined on distinct probability spaces. We obtain the following result.
\begin{theorem}\label{thm:uniqlaw}
The weak-strong uniqueness in law holds true for the stochastic Euler equations \eqref{eq:Euler} in the following sense: Let
$$\left[ (\Omega^1,\mathfrak F^1,(\mathfrak F^{1}_{t})_{t\geq0},\mathbb P^1),\bfu^1,\mathcal V^1 , W^1 \right]$$
be a dissipative martingale solution to \eqref{eq:Euler} in the sense of Definition \ref{def:soleuler} and let $\bfu^2$ be a strong solution of the same problem in the sense of Definition \ref{def:strsol} (with $\mathfrak t=T$) defined on a stochastic basis $(\Omega^2,\mathfrak F^2,(\mathfrak F^{2}_{t})_{t\geq0},\prst^2)$ with the Wiener process $W^2$. Suppose that
$$\prst^1\circ (\bfu^1(0))^{-1}=\prst^2\circ (\bfu^2(0))^{-1},$$
then
\begin{equation}\label{law}
\prst^1\circ(\bfu^1,\mathcal V^1)^{-1}=\prst^2\circ(\bfu^2,(\delta_{\bfu^2(t,x),0,0}))^{-1}.
\end{equation}
\end{theorem}
\begin{proof}
Let us assume that
$$\left[ (\Omega^1,\mathfrak F^1,(\mathfrak F^{1}_{t})_{t\geq0},\mathbb P^1),\bfu^1,\mathcal V^1 , W^1 \right]$$
is a dissipative martingale solution to \eqref{eq:Euler} in the sense of Definition \ref{def:soleuler} and let $\bfu^2$ be a strong solution of the same problem in the sense of Definition \ref{def:strsol} (with $\mathfrak t=T$). Different to Theorem \ref{thm:uniq} $\bfu^2$ is now defined on a distinct stochastic basis $(\Omega^2,\mathfrak F^2,(\mathfrak F^{2}_{t})_{t\geq0},\prst^2)$ with a distinct Wiener process $W^2$.
We set $\bfv^j=\bfu^j-\bfu^j(0)$ for $t\geq0$ and $j=1,2$. We consider the topological space
\begin{align*} \mathcal X:= C([0,T];W^{-4,2}_{\Div}(\mt))\cap C_w([0,T];L^2_{\Div}(\mt))\times Y_2(Q,\R^3)\times C([0,T],\mathfrak U_0)
 \end{align*}
together with the $\sigma$-algebra $\mathscr B_{\mathcal X}$ as defined in \eqref{eq:BX}. Setting
\begin{align*}
\Theta=L^2(\mt)\times\mathcal X,\quad \mathscr B_{\Theta}=\mathscr B(L^2(\mt))\otimes\mathscr B_{\mathcal X}
\end{align*} 
 we denote the probability law $\mathcal L[\bfu^j(0),\bfv^j,\mathcal V^j,W^j]$ on $(\Theta,\mathscr B_{\Theta})$ by $\mu^j$ (recall that $\mathcal V^2=(\delta_{\bfu^2(t,x),0,0})$ for the strong solution). It
satisfies
\begin{align*}
\mu^j(\mathscr A)=\mathbb P_j\big([\bfu^j(0),\bfv^j,\mathcal V^j,W^j]\in\mathscr A\big),\quad \mathscr A\in \mathscr B_{\Theta}.
\end{align*}
The generic element of $\Theta$ is denoted by $\theta=(\tilde \bfu_0,\tilde W,\tilde\bfv,\tilde{\mathcal{V}})$. The marginal of each $\mathbb P_j$ on the $\tilde\bfu_0$-coordinate is $\Lambda$, the marginal on the $\tilde W$-coordinate is the Wiener measure $\mathbb P_*$ and the distribution of the pair is the product measure $\Lambda\otimes \mathbb P_*$ because $\bfu_0^j$ is $\mathfrak F^j_0$-measurable and $W^j$ is independent of $\mathfrak F^j_0$.
Moreover, under $\mathbb P_j$ the initial value of the $\tilde \bfv$-coordinate is zero a.s.\\
In a first step we are going to construct a product probability
space. In order to do this we need regular conditional probabilities and in the following we argue why this is possible in our situation.  
Let $(\mathcal O,\mathscr Y)$ be a measure space, where $\mathcal O$ is a Hausdorff topological space
and $\mathscr Y$ is countably generated. Let $\mathcal U$ be a  regular probability measure on
$(\mathcal O,\mathscr Y)$, i.e.
\begin{align*}
\mathcal U(A)=\sup\{\mathcal U(K):\,K\subset A\,\text{compact}\}\quad\forall A\in \mathscr Y.
\end{align*}
In other words $\mathcal U$ is Radon.
It is well-known that under these assumptions there is a regular conditional probability for $\mathcal U$, see e.g. \cite[introduction]{HoJo}. 
Since $\mathcal X$ is a quasi-Polish space and $L^2(\mt)$ is a Banach space it is clear that $\Theta$ is Hausdorff. We have to argue that $B_{\Theta}$ is countable generated. It is clear that $(\mathscr B(C([0,T];W^{-4,2}_{\Div}(\mt))$ and $\mathscr B(C([0,T],\mathfrak U_0))$ are countably generated since the spaces in question are both Polish.
As far as $\mathscr B_T(C_w([0,T];L^2_{\Div}(\mt))\big)$ is concerned we refer to \cite[Section 4]{bos} for a corresponding statement. Finally, since the function $f_n$ from \eqref{eq:quasi} range in the Polish space $[-1,1]$ and are continuous we have that $\sigma(f_n)$ is countably generated for each $n\in\N$. Since the family $\{f_n\}$ is countable we conclude that $\mathscr B_{Y}$ defined in \eqref{eq:sigmaY} is countably generated.
In conclusion there is a regular conditional probability
\begin{align*}
Q_j(\tilde\bfu_0,\tilde W,\mathscr A):L^2(\mt)\times C([0,T],\mathfrak U_0)\times \mathscr B_{\bfu}\otimes \mathscr B_{Y}\rightarrow[0,1]
\end{align*}
such that
\begin{enumerate}
\item[(i)] For each $(\tilde\bfu_0,\tilde W)\in L^2(\mt)\times C([0,T],\mathfrak U_0)$ we have that $$Q_j(\bfw,B,\cdot):(C([0,T];W^{-4,2}_{\Div}(\mt)\cap C_w([0,T];L^2(\mt));\mathscr B_{\bfu}\otimes \mathscr B_{Y})\rightarrow[0,1]$$
is a probability measure;
\item[(ii)] The mapping 
$(\tilde\bfu_0,\tilde W)\rightarrow Q_j(\tilde\bfu_0,\tilde W,\mathscr A)$ is $\mathscr B(L^2(\mt))\otimes \mathscr B\big(C([0,T],\mathfrak U_0)\big)$ measurable for each $\mathscr A\in\mathscr B_{\bfu}\otimes \mathscr B_{Y}$;
\item[(iii)] We have that
\begin{align*}%\label{eq:cond'}
\mathbb P_j(G\times \mathscr A)=&\int_GQ_j(\tilde\bfu_0,\tilde W,\mathscr A)\,\dd \mu(\tilde\bfu_0)\,\dd\mathbb P_*(\tilde W),\quad \mathscr A\in\mathscr B_{\bfu}\otimes \mathscr B_{Y},
\end{align*}
for all $G\in \mathscr B(L^2(\mt))\otimes \mathscr B\big(C([0,T],\mathfrak U_0)$.
\end{enumerate}
Finally we define
\begin{align*}
\tilde\Omega=\Theta\times C([0,T];W^{-4,2}_{\Div}(\mt))\cap C_w([0,T];L^2_{\Div}(\mt))\times Y_2(Q,\R^3)
\end{align*}
and $\tilde{\mathfrak F}$ is the completion of 
$\mathscr B_{\Theta}\otimes \mathscr B_{\bfu}\otimes \mathscr B_{Y}$
with respect to the probability measure
\begin{align*}
\tilde{\mathbb{P}}(G\times \mathscr A_1\times\mathscr A_2)=&\int_GQ_j(\tilde\bfu_0,\tilde W,\mathscr A_1)Q_j(\tilde\bfu_0,\tilde W,\mathscr A_2)\,\dd \mu(\tilde\bfu_0)\,\dd\mathbb P_*(\tilde W)
\end{align*}
for $\mathscr A_1,\mathscr A_2\in\mathscr B_{\bfu}\otimes \mathscr B_{Y}$ and $G\in \mathscr B(L^2(\mt))\otimes \mathscr B\big(C([0,T],\mathfrak U_0)$.
The space $(\tilde\Omega,\tilde{\mathfrak{F}},\tilde{\mathbb{P}})$ is the product probability space we were seeking and we obtain
\begin{align*}%\label{eq:2205}
\tilde{\p}\big(\big\{\tilde\omega\in\tilde\Omega:\,(\tilde\bfu_0,\Tilde W,\tilde\bfv^j,\tilde{\mathcal V}^j)\in\mathscr A\big\}\big)=\mu_j(\mathscr A),\quad \mathscr A\in\mathscr B_\Theta,\quad j=1,2.
\end{align*}
Finally, we define the filtrations
\begin{align*}
\tilde{\mathfrak F}^j_t&=\sigma\Big(\sigma \big(\tilde\bfu_0,\bfr_t\tilde{W},\bfr_t\tilde{\bfv}^j,\bfr_t\tilde{\mathcal V}^j\big)\cup\sigma_[\tilde{\mathcal V}^j]\cup\big\{\mathcal N\in\tilde{\mathfrak F};\;\tilde{\p}(\mathcal N)=0\big\}\Big),\; j=1,2,\\
\tilde{\mathfrak F}_t&=\sigma\Big(\sigma \big(\tilde\bfu_0,\bfr_t\tilde{W},\bfr_t\tilde{\bfv}^1,\bfr_t\tilde{\bfv}^2\big)\cup\sigma_t[\tilde{\mathcal V^1}]\cup\sigma_t[\tilde{\mathcal V}^2]\cup\big\{\mathcal N\in\tilde{\mathfrak{F}};\;\tilde{\p}(\mathcal N)=0\big\}\Big),
\end{align*}
which ensure the correct measurabilities. Here $\sigma_t$ denotes the history of a random distribution as defined in \eqref{eq:2402a}, where generalised Young measures are identified as random distribution in the sense of \eqref{eq:2402b}.\\
In the next step we aim to show that for $j=1,2$
$$\left[ (\Omega,\tilde{\mathfrak F},(\tilde{\mathfrak F}_{t})_{t\geq0},\tilde{\mathbb P}),\Tilde\bfv^j+\tilde\bfu_0,\tilde{\mathcal V}^j , \tilde W \right]$$
is a dissipative martingale solution to \eqref{eq:Euler} in the sense of Definition \ref{def:soleuler} and that $\bfv^2+\bfu_0$ is a strong solution. As in \eqref{eq:2011} we can prove
that $\tilde{\p}$-a.s.
\begin{align*}
\tilde\bfv^j(t,x)+\tilde\bfu_0(x)=\langle \tilde\nu^j_{t,x},\bfxi\rangle\quad\text{for a.a.}\quad (t,x)\in Q_T,
\end{align*}
where $\mathcal V^j=(\tilde\nu^j_{t,x},\tilde\nu_{t,x}^{\infty,j},\tilde\lambda^j)$.
Defining the functional
\begin{align*}
\mathfrak M(\bfw,\mathcal V)_t&=\int_{\mt}\bfw(t)\cdot\bfvarphi\dx
-\int_0^t\int_{\mt}\big\langle\nu_{t,x},\bfxi\otimes\bfxi\big\rangle:\nabla\bfvarphi\dx\,\dif s\\&-\int_{(0,t)\times\mt}\big\langle\nu_{t,x}^\infty,\bfxi\otimes\bfxi\big\rangle:\nabla\bfvarphi\,\dif \lambda,\quad\bfphi\in C^\infty_{\Div}(\mt),
\end{align*}
we can argue as in Section \ref{subsec:new} to prove that
$\mathfrak M(\tilde\bfu^j+\tilde\bfu_0,\mathcal V^j)$ is an $(\tilde{\mathfrak F}^j_t)$-martingale. Moreover its quadratic variation and cross variation with respect to $\tilde W$ are given by 
$\mathfrak N$ and $\mathfrak N^k$ respectively. Consequently, both solutions satisfy the momentum equation in the sense of Definition \ref{def:soleuler} (g) driven by $\tilde W$. Finally, we can use
again Proposition \ref{RDT3} to argue that the energy inequality continuous to hold
on the product probability space following the arguments of Section \ref{subsec:new}.\\
% the functionals
%\[
%\begin{split}
%\mathscr S^n(\mathcal V)_\tau&= \frac{1}{n}\big[E_\tau(\mathcal V)\big]^n  - \frac{2n-1}{2} \int_0^\tau [E_t(\mathcal V)]^{n-1} \|\Phi\|_{L_2(\mathfrak U,L^2(\mt))}^2
% \dt,\\ E(\mathcal V)_t&=\frac{1}{2}\int_ {\mt}\frac{1}{2}\big\langle\nu_{t,x},|\bfxi|^2\big\rangle\dx+\frac{1}{2}\lambda_t(\mt)
%\end{split}
%\]
%to prove that $\mathscr S^n(\tilde{\mathcal V}^j)_t$ is an a.s. $(\tilde{\mathfrak F}^j_t)$-supermartingale. \textcolor{blue}{It seems we need instead $\sigma(\tilde\bfu^j,\tilde{\mathcal V}^j(s),\,0\leq s\leq t)$}\\
In order to apply our pathwise weak-strong uniqueness result from Theorem \ref{thm:uniq} it suffices to argue that $\tilde{\bfu}^2=\tilde\bfv^2+\tilde\bfu_0$ is a strong solution. On the original probability space $(\Omega^2,\mathfrak F^2,\p^2)$ the strong solution $\bfu^2$ is supported on $C([0,T];C^1(\mt))$ and we have
$\mathcal V^2=(\delta_{\bfu^2(t,x)},0,0)$ $\p^2$-a.s. The embedding
\begin{align*}
C([0,T];C^1(\mt))\hookrightarrow C([0,T];W^{-4,2}(\mt))
\end{align*}
is continuous and dense such that 
\begin{align*}
C([0,T];C^1(\mt))\in \mathscr B(C([0,T];W^{-4,2}(\mt))\big)\subset \mathscr B_{\bfu},
\end{align*}
cf. \cite[Cor. A.2]{on2}. We conclude
\begin{align*}
\mu_2(C([0,T];C^1(\mt)))=\mathbb P_2\big(\bfv\in C([0,T];C^1(\mt))\big)=1
\end{align*} 
such that $\tilde\bfu^2$ is a strong solution on in the sense of Definition \ref{def:strsol} (with $\mathfrak t=T$) on
$(\tilde\Omega,\tilde{\mathfrak F},\tilde{\p})$. Moreover, we have $\tilde\bfu^2(0)=\tilde\bfu_0=\tilde\bfu^1(0)$ $\tilde{\p}$-a.s. We conclude by Theorem \ref{thm:uniq} (see also Remark \ref{rem:new})
that $\tilde{\p}$-a.s.
\begin{align*}
(\tilde\bfu^1,\tilde{\mathcal V}^1)=(\tilde\bfu^2,\tilde{\mathcal V}^2)=(\tilde\bfu^2,(\delta_{\tilde\bfu^2},0,0))
\end{align*}
Finally, we obtain
\begin{align*}
\mu_1(\mathscr A)&=\tilde{\p}\big(\big\{\tilde\omega\in\tilde\Omega:\,(\tilde\bfu_0,\tilde W,\tilde\bfv^1+\tilde\bfu_0,\tilde{\mathcal V}^1)\in\mathscr A\big\}\big)\\
&=\tilde{\p}\big(\big\{\tilde\omega\in\tilde\Omega:\,(\tilde\bfu_0,\tilde W,\tilde\bfv^2+\tilde\bfu_0,\tilde{\mathcal V}^2)\in\mathscr A\big\}\big)=\mu_2(\mathscr A)
\end{align*}
for all $\mathscr A\in \mathscr B_{\Theta}$ which finishes the proof.
\end{proof}

\subsection*{Acknowledgement}
The authors would like to thank M. Hofmanov\'a, R. Zhu and X. Zhu for stimulating discussions and suggestions which helped to improve the paper.\\
The second author is financed by the the Maxwell Institute Graduate School in Analysis \& its Applications in Edinburgh.

\subsection*{Conflict of Interest} The authors declare that they have no conflict of interest.

\end{document}